\documentclass[11pt,reqno]{amsart}
\usepackage{amsmath,amsthm,amssymb,mathrsfs,tocvsec2}
\usepackage{mathtools,bm}
\usepackage[rmargin=1in,lmargin=1in]{geometry}
\usepackage[breaklinks=true]{hyperref}
\usepackage[capitalize]{cleveref}
\usepackage{aliascnt}

\usepackage{color,soul}
\allowdisplaybreaks

\numberwithin{equation}{section}
\theoremstyle{plain}
\newtheorem{theorem}{Theorem}[section]
\newtheorem*{theorem*}{Theorem}
\newtheorem{conj}{Conjecture}

\newaliascnt{lemma}{theorem}
\newtheorem{lemma}[lemma]{Lemma}
\aliascntresetthe{lemma}
\newaliascnt{proposition}{theorem}
\newtheorem{proposition}[proposition]{Proposition}
\aliascntresetthe{proposition}
\newaliascnt{corollary}{theorem}
\newtheorem{corollary}[corollary]{Corollary}
\aliascntresetthe{corollary}

\theoremstyle{definition}

\newaliascnt{remark}{theorem}
\newtheorem{remark}[remark]{Remark}
\aliascntresetthe{remark}
\newaliascnt{example}{theorem}
\newtheorem{example}[example]{Example}
\aliascntresetthe{example}

\Crefname{lemma}{Lemma}{Lemmas}
\Crefname{proposition}{Proposition}{Propositions}
\Crefname{corollary}{Corollary}{Corollaries}
\Crefname{conj}{Conjecture}{Conjectures}
\Crefname{remark}{Remark}{Remarks}
\Crefname{example}{Example}{Examples}

\def\geq{\geqslant}
\def\leq{\leqslant}
\renewcommand{\subset}{\subseteq}
\newcommand{\dif}{\mathop{}\!\mathrm{d}}
\def\dx{\dif x}
\def\dy{\dif y}

\newcommand{\poch}[2]{(#1)_{{#2}}}
\newcommand{\qpoch}[3]{(#1;#2)_{{#3}}}
\newcommand{\qnum}[2][q]{[#2]_#1}
\newcommand{\tempbinom}{\genfrac{[}{]}{0pt}{}}
\newcommand{\qbinom}[3][q]{\tempbinom{#2}{#3}_{#1}}
\let\=\relax\newcommand{\=}{\mathrel{\phantom{=}}}
\newcommand{\wm}{\mathrel{\succcurlyeq_\mathrm{w}}}
\newcommand{\m}{\mathrel{\succcurlyeq}}
\newcommand\R{\mathbb{R}}
\newcommand\Q{\mathbb{Q}}
\newcommand\Z{\mathbb{Z}}
\newcommand\F{\mathbb{F}}
\def\fp{\mathbb F_{\geqslant0}}

\def\kp{\mathbb K_{\geqslant0}}

\newcommand*{\SetSuchThat}[1][]{} 

\newcommand*{\MvertSets}{%
	\renewcommand*\SetSuchThat[1][]{%
		\mathclose{}%
		\nonscript\;##1\vert\penalty\relpenalty\nonscript\;%
		\mathopen{}%
	}%
}
\MvertSets 
\DeclarePairedDelimiterX \Set [2] {\lbrace}{\rbrace}
{\,#1\SetSuchThat[\delimsize]#2\,}

\newcommand{\mydef}[1]{\textbf{#1}}

\begin{document}
	
	\title{Majorization via positivity of Jack and Macdonald polynomial
		differences}
	
	\author{Hong Chen}
	\address[H.~Chen]{Department of Mathematics, Rutgers University,
		Piscataway 08854, USA}
	\email{\tt hc813@math.rutgers.edu}
	
	\author{Apoorva Khare}
	\address[A.~Khare]{Department of Mathematics, Indian Institute of
		Science, Bangalore -- 560012, India; and Analysis and Probability
		Research Group, Bangalore -- 560012, India}
	\email{\tt khare@iisc.ac.in}
	
	\author{Siddhartha Sahi}
	\address[S.~Sahi]{Department of Mathematics, Rutgers University,
		Piscataway 08854, USA}
	\email{\tt sahi@math.rutgers.edu}
	
	\date{\today}
	
	

	\begin{abstract}
		Majorization inequalities have a long history, going back
		to Maclaurin and Newton. They were recently studied for
		several families of symmetric functions, including by
		Cuttler--Greene--Skandera (2011), Sra (2016), Khare--Tao
		(2021), McSwiggen--Novak (2022), and Chen--Sahi (2024+)
		among others. Here we extend the inequalities by these
		authors to Jack and Macdonald polynomials, and obtain
		conjectural characterizations of majorization and of weak
		majorization of the underlying partitions. We prove these
		characterizations for two variables.
		In fact, we upgrade---and prove in the above cases---the
		characterization of majorization, to containment of Jack and Macdonald
		differences lying in the Muirhead semiring.
	\end{abstract}
	
	\maketitle

	\settocdepth{section}
	\tableofcontents
	
	\section{Introduction}
	
	In this section, we first recall some of the history of symmetric
	functions and of inequalities associated to them.
	
	\subsection{Symmetric functions}
	
	Symmetric polynomials and symmetric functions have a long history and are
	ubiquitous in mathematics and physics, especially in algebra and
	representation theory, combinatorics, probability, and
	statistics.
	We give a very brief history of the symmetric functions below. For more details, see \cite{KS06}.
	
	One of the most classical results is the relation between the
	\mydef{elementary symmetric polynomials} and the \mydef{power sums},
	studied by Newton around 1666. These relations, together
	with Vieta's formulas, enable one to find the power sums of the roots of
	a univariate polynomial, without finding the roots explicitly.
	
	Since the work of Jacobi in the mid-nineteenth century and the
	papers of Frobenius, Schur, Weyl, MacMahon, and
	Young in the early twentieth century, \mydef{Schur polynomials} have been
	well studied, and play a significant role in the representation theory of
	the symmetric group $S_n$ and the complex general linear group
	$GL_n(\mathbb C)$. 
	
	Around 1960, Hall and Littlewood independently introduced a one-parameter
	generalization of the Schur functions, now known as the
	\mydef{Hall--Littlewood polynomials}. By the work of Green and Macdonald,
	these functions are related to the representation theory of $GL_n$ over
	finite and $p$-adic fields.
	
	In the late 1960s, Jack discovered \mydef{Jack polynomials}, unifying
	Schur polynomials and \mydef{zonal polynomials}---the latter being
	related to the representation theory of $GL_n(\mathbb R)$ and multivariate
	statistics.
	
	Hall--Littlewood polynomials and Jack polynomials are quite different
	generalizations of Schur polynomials. In the 1980s, Macdonald unified
	these developments into a two-parameter family of symmetric polynomials,
	the \mydef{Macdonald polynomials}. Hall--Littlewood polynomials can be
	obtained from Macdonald polynomials by specializing one of the parameters
	to~0; while Jack polynomials arise as the limiting case when both
	parameters approach~1.
	
	\subsection{Symmetric function inequalities}
	
	The notion of positivity has an even longer history: it begin with
	mathematics itself---counting and measuring objects. Let us restrict our
	attention to inequalities about symmetric functions. 
	
	Most families of symmetric polynomials are indexed by integer partitions.
	A \mydef{partition} (of length at most $n$) is an $n$-tuple $\lambda=(\lambda_1,\dots,\lambda_n)$ of integers such that $\lambda_1\geq\cdots\geq\lambda_n\geq0$. 
	Denote the set of such partitions by $\mathcal P_n$.
	Let $|\lambda|\coloneqq\lambda_1+\cdots+\lambda_n$ be the size of $\lambda$.
	
	Recall that for partitions $\lambda$ and $\mu$ in $\mathcal P_n$, one has the following partial orders:
	\begin{itemize}
		\item $\lambda$ \mydef{contains} $\mu$, denoted by $\lambda\supseteq\mu$, if $\lambda_i\geq\mu_i$ for $i=1,\dots,n$;
		\item $\lambda$ \mydef{weakly majorizes} $\mu$, denoted by $\lambda\wm\mu$, if $\displaystyle\sum_{i=1}^k \lambda_i \geq\sum_{i=1}^k \mu_i$, for $k=1,\dots,n$;
		\item $\lambda$ \mydef{majorizes} $\mu$, denoted by $\lambda\m\mu$, if $\lambda\wm\mu$ and $|\lambda|=|\mu|$.
	\end{itemize}
	Throughout the paper, $\bm1=\bm1_n\coloneqq(1,\dots,1)$ ($n$ times).
	
	We now recall some symmetric function inequalities related to the three
	partial orders. The notation used here and below is found in
	Section~\ref{Sprelim}.
	
	\subsubsection{Containment-type and expansion positivity}
	
	Of great importance in representation theory and algebraic combinatorics is the notion of \mydef{expansion positivity}.
	
	We say a symmetric polynomial is \mydef{Schur positive} if it can be written as a sum of Schur polynomials with non-negative coefficients.
	For example, by the celebrated Littlewood--Richardson rule, the product of two Schur polynomials is Schur positive:
	$\displaystyle s_\lambda\cdot s_\mu = \sum_{\nu} c_{\lambda\mu}^\nu s_\nu$,
	where the Littlewood--Richardson coefficients $c_{\lambda\mu}^\nu\geq0$.
	Similarly, the Schur polynomial $s_\lambda$ is \mydef{monomial positive}:
	$\displaystyle s_\lambda = \sum_{\mu\m\lambda} K_{\lambda\mu}m_\mu$,
	where the Kostka numbers $K_{\lambda\mu}\geq0$.
	
	Expansion positivity naturally suggests underlying combinatorial objects or structure. 
	For example, Littlewood--Richardson coefficients have various combinatorial interpretations: Littlewood--Richardson tableaux \cite{Stanley,Mac15}, Berenstein–Zelevinsky patterns \cite{BZ92} and Knutson–Tao honeycombs \cite{KT99}; and Kostka numbers count the number of semi-standard Young tableaux \cite{Stanley,Mac15}.
	
	For Jack polynomials $P_\lambda^{(\alpha)}(x)$, it was conjectured in Macdonald's book \cite[VI.(10.26?)]{Mac15} and proved by Knop and Sahi \cite{KS} that their expansions in terms of the monomial basis $(m_\lambda)$ are positive (and integral, under certain normalizations) in the parameter $\alpha$.
	
	The following binomial formula for Jack polynomials is due to
	Okounkov and Olshanski \cite{OO97}:
	\begin{align}
		\frac{P_\lambda^{(\alpha)}(x+\bm1)}{P_\lambda^{(\alpha)}(\bm1)} &= \sum_{\nu\subseteq\lambda}\binom{\lambda}{\nu}_\alpha \ \frac{P_\nu^{(\alpha)}(x)}{P_\nu^{(\alpha)}(\bm1)}.
	\end{align}
	The coefficient $\binom{\lambda}{\nu}_\alpha$ is known as the \mydef{generalized binomial coefficient}.
	Sahi \cite{Sahi-Jack} showed that $\binom{\lambda}{\nu}_\alpha$ is non-negative (i.e.\ lies in a positive cone $\fp$ defined in \cref{eqn:fp}) and hence $\frac{P_\lambda^{(\alpha)}(x+\bm1)}{P_\lambda^{(\alpha)}(\bm1)}$ is \mydef{Jack positive}.
	Recently, Chen and Sahi proved \cite{CS24} that the coefficient $\binom{\lambda}{\nu}_\alpha$ is monotone in $\lambda$:
	\begin{align}
		\binom{\lambda}{\nu}_\alpha-\binom{\mu}{\nu}_\alpha\in\fp,	\quad  \lambda\supseteq\mu,
	\end{align}
	which, in turn, leads to the Jack positivity of $\frac{P_\lambda^{(\alpha)}(x+\bm1)}{P_\lambda^{(\alpha)}(\bm1)}-\frac{P_\mu^{(\alpha)}(x+\bm1)}{P_\mu^{(\alpha)}(\bm1)}$ for $\lambda \supseteq \mu$:
	\begin{align}
		\frac{P_\lambda^{(\alpha)}(x+\bm1)}{P_\lambda^{(\alpha)}(\bm1)} -\frac{P_\mu^{(\alpha)}(x+\bm1)}{P_\mu^{(\alpha)}(\bm1)} 
		= \sum_{\nu\subseteq\lambda} \left(\binom{\lambda}{\nu}_\alpha-\binom{\mu}{\nu}_\alpha\right) \frac{P_\nu^{(\alpha)}(x)}{P_\nu^{(\alpha)}(\bm1)}.
	\end{align}
	
	The study of generalized binomial coefficients originated in two distinct contexts. In statistics, their zonal case ($\alpha=2$) was first examined by Constantine in the 1960s, while Lascoux~\cite{Lascoux} explored their Schur case ($\alpha=1$) to compute Chern classes for exterior and symmetric squares of vector bundles (see also \cite[p.~47, Example 10]{Mac15}). 
	Seminal contributions emerged in the 1990s through the foundational work of Lassalle~\cite{Las90, Las98}, Kaneko~\cite{Kan93}, and Okounkov--Olshanski~\cite{OO-schur, OO97}.
	
	Generalized binomial coefficients can be explicitly computed via
	evaluating \mydef{interpolation Jack polynomials}, first studied
	by Knop--Sahi~\cite{Sahi94,KS96} and Okounkov--Olshanski~\cite{OO97}. Recently,
	the $BC$-symmetry analogue of binomial coefficients and interpolation
	polynomials were studied in~\cite{CS24}.
	
	In what follows, we write Jack polynomials with parameter $\tau=1/\alpha$, that is, our $P_\lambda(x;\tau)$ is equal to Macdonald's $P_\lambda^{(1/\tau)}(x)$.
	
	\begin{theorem}[{\cite[Theorem 6.1]{CS24}}]\label{thm:contain-J}
		The following are equivalent for partitions $\lambda, \mu$:
		\begin{enumerate}
			\item The difference 
			$\displaystyle
			\frac{P_\lambda(x+\bm1;\tau)}{P_\lambda(\bm1;\tau)}-\frac{P_\mu(x+\bm1;\tau)}{P_\mu(\bm1;\tau)}$ is \mydef{Jack positive} over $\fp$, namely, can be written as an
			$\fp$-combination of $P_\nu(x;\tau)$.
			\item For every $\tau_0\in[0,\infty]$, the difference $\displaystyle \frac{P_\lambda(x+\bm1;\tau_0)}{P_\lambda(\bm1;\tau_0)}-\frac{P_\mu(x+\bm1;\tau_0)}{P_\mu(\bm1;\tau_0)}$ is \mydef{$\tau_0$-Jack positive} over $\R_{\geq0}$, namely, can be written as an $\mathbb R_{\geq0}$-combination of $P_\nu(x;\tau_0)$.
			\item For some $\tau_0\in[0,\infty]$, the difference $\displaystyle \frac{P_\lambda(x+\bm1;\tau_0)}{P_\lambda(\bm1;\tau_0)}-\frac{P_\mu(x+\bm1;\tau_0)}{P_\mu(\bm1;\tau_0)}$ is \mydef{$\tau_0$-Jack positive} over $\R_{\geq0}$, namely, can be written as an $\mathbb R_{\geq0}$-combination of $P_\nu(x;\tau_0)$.
			\item $\lambda$ contains $\mu$, $\lambda\supseteq\mu$.
		\end{enumerate}
	\end{theorem}
	
	It is well-known, see \cite[Chapters I, VI, VII]{Mac15}, that Jack polynomials $P_\lambda(x;\tau)$ specialize to many other symmetric polynomials: 
	monomial symmetric polynomials $m_\lambda$ when $\tau=0$,
	zonal polynomials $Z_\lambda$ when $\tau=1/2$ (spherical functions for $GL_n(\mathbb{R})/O_n$) and
	$\tau=2$ (for $GL_n(\mathbb{H})/Sp_n$),
	Schur polynomials $s_\lambda$ when $\tau=1$,
	and elementary symmetric polynomials $e_{\lambda'}$ when $\tau=\infty$
	(where $\lambda'$ is the conjugate of $\lambda$).
	Hence, we also have the following characterization, which somewhat
	surprisingly also holds for power sums:
	
	\begin{theorem}[{\cite[Theorems 6.2, 6.3]{CS24}}]\label{thm:contain-J0}
		The following are equivalent for partitions $\lambda, \mu$:
		\begin{enumerate}
			\item $\lambda$ contains $\mu$.
			\item The difference $\displaystyle
			\frac{m_\lambda(x+\bm1)}{m_\lambda(\bm1)}-\frac{m_\mu(x+\bm1)}{m_\mu(\bm1)}$
			is \mydef{monomial positive}.
			\item The difference $\displaystyle
			\frac{Z_\lambda(x+\bm1)}{Z_\lambda(\bm1)}-\frac{Z_\mu(x+\bm1)}{Z_\mu(\bm1)}$
			is \mydef{zonal positive}.
			\item The difference 
			$\displaystyle
			\frac{s_\lambda(x+\bm1)}{s_\lambda(\bm1)}-\frac{s_\mu(x+\bm1)}{s_\mu(\bm1)}$
			is \mydef{Schur positive}.
			\item The difference $\displaystyle \frac{e_\lambda'(x+\bm1)}{e_\lambda'(\bm1)}-\frac{e_\mu'(x+\bm1)}{e_\mu'(\bm1)}$ is \mydef{elementary positive}.
			\item The difference $\displaystyle \frac{p_\lambda(x+\bm1)}{p_\lambda(\bm1)}
			- \frac{p_\mu(x+\bm1)}{p_\mu(\bm1)}$ is \textbf{power sum positive}.
		\end{enumerate}
	\end{theorem}

	\subsubsection{Majorization-type and evaluation positivity}
	
	As the name suggests, evaluation positivity is about evaluating
	real polynomials, which is more commonly seen in classical
	analysis. We begin with the special case of majorization:
	$(k+1,k-1)\m(k,k)$ for $k\geq1$.
	
	It is a classical result, dating back to Newton, that polynomials with non-positive real roots have \mydef{ultra log-concave} coefficients. More precisely, for $$f(t)=\prod_{i=1}^n (t+x_i)=\sum_{k=0}^n e_k(x)\, t^{n-k},$$
	Newton’s inequality asserts  
	\begin{align}\label{eqn:Newton}
		\frac{e_{k+1}(x)}{e_{k+1}(\bm1)} \cdot \frac{e_{k-1}(x)}{e_{k-1}(\bm1)} 
		\leq \left(\frac{e_k(x)}{e_k(\bm1)}\right)^2, \qquad x\in[0,\infty)^n.
	\end{align}
	
	Equivalently, the \textit{normalized} elementary symmetric polynomials
	$a_k={e_k(x)}/{e_k(\bm1)}$ form a \mydef{log-concave} sequence:
	$a_{k+1}a_{k-1}\leq a_k^2$. A familiar instance of log-concavity is given
	by the binomial coefficients  
	$\displaystyle \binom{n}{0},\binom{n}{1},\dots,\binom{n}{n}$.
	Log-concavity also plays a central role in combinatorial Hodge theory, as developed by Adiprasito--Huh--Katz~\cite{AHK}.

	In 1959, long after Newton, Gantmacher \cite{Gantmacher} proved a similar
	inequality for power sums:
	\begin{align}
		p_{k+1}(x)p_{k-1}(x) \geq p_k(x)^2,	\quad x\in[0,\infty)^n.
	\end{align}
	Note that $p_k(\bm1_n)=n$ for any $k$, and hence one can normalize
	Gantmacher's inequality in the same way as Newton's.
	
	For partitions $\lambda$ and $\mu$ with $|\lambda|=|\mu|$, Muirhead showed in \cite{Muirhead} that for the monomial symmetric polynomial $m_\lambda$
	\begin{align}\label{EMuirhead}
		\frac{m_\lambda(x)}{m_\lambda(\bm1)}\geq \frac{m_\mu(x)}{m_\mu(\bm1)}, 	\quad x\in [0,\infty)^n
	\end{align}
	if and only if $\lambda\m\mu$. (Note that~\eqref{EMuirhead} holds even
	when $\lambda\in\R^n$, in which case, $m_\lambda$ is defined as the
	orbit-sum of $x^\eta=x_1^{\eta_1}\cdots x_n^{\eta_n}$, where $\eta$ runs
	over the orbit $S_n\cdot\lambda$, acting by permutation.)
	
	Cuttler, Greene, and Skandera \cite{CGS} generalized the inequalities of Newton and Gantmacher to elementary symmetric polynomials and power sums indexed by partitions, and conjectured an inequality for Schur polynomials, which
	was later proved by Sra \cite{Sra} (again, for $\lambda, \mu$ real
	$n$-tuples). We recall all these results.
	
	\begin{theorem}\label{thm:majorization}
		Let $\lambda$ and $\mu$ be integer partitions with $|\lambda|=|\mu|$.
		The following are equivalent:
		\begin{enumerate}
			\item $\lambda$ majorizes $\mu$, $\lambda\m\mu$.
			\item (\cite{Muirhead}, Muirhead's inequality) The following difference is positive:
			\begin{align}
				\frac{m_\lambda(x)}{m_\lambda(\bm1)} - \frac{m_\mu(x)}{m_\mu(\bm1)}\geqslant0, \quad \forall x\in[0,\infty)^n.
			\end{align}
			\item (\cite[Theorem 3.2]{CGS}, generalized Newton's inequality) The following difference is positive:
			\begin{align}
				\frac{e_{\lambda'}(x)}{e_{\lambda'}(\bm1)} - \frac{e_{\mu'}(x)}{e_{\mu'}(\bm1)}\geqslant0, \quad \forall x\in[0,\infty)^n.
			\end{align}
			\item (\cite[Theorem 4.2]{CGS}, generalized Gantmacher's inequality) The following difference is positive:
			\begin{align}
				\frac{p_\lambda(x)}{p_\lambda(\bm1)} - \frac{p_\mu(x)}{p_\mu(\bm1)}\geqslant0, \quad \forall x\in[0,\infty)^n.
			\end{align}
			\item (\cite[Conjecture~7.4, Theorem~7.5]{CGS} and \cite{Sra}, Cuttler--Greene--Skandera and Sra's inequality) The following difference is positive:
			\begin{align}
				\frac{s_\lambda(x)}{s_\lambda(\bm1)} - \frac{s_\mu(x)}{s_\mu(\bm1)}\geqslant0, \quad \forall x\in[0,\infty)^n.
			\end{align}
		\end{enumerate}
	\end{theorem}
	
	Since Jack polynomials and Macdonald polynomials generalize monomial
	symmetric, elementary symmetric, and Schur polynomials, the first and
	third authors made some similar conjectures in~\cite{CS24}. In this
	paper, we extend these conjectures and prove them in some cases.
	
	As we will see below---as also in \cite{CGS,KT,MN,Sra}---the
	more challenging implication is to show that if $\lambda \m \mu$ then the
	normalized differences of various symmetric polynomials are non-negative
	on $[0,\infty)^n$. 
	An alternate, stronger approach would be to show that these differences ``decompose'' non-negatively into non-negative quantities, like monomials and differences of monomials as in Muirhead's inequality. This is known for $e_{\lambda'}$ and for $p_\lambda$, as we now state. 	
	
	Let $M_\lambda\coloneqq m_\lambda/m_\lambda(\bm1)$ be the normalized
	monomial. Recall \cite[Section 6]{CGS} that the \mydef{Muirhead cone}
	$\mathcal M_C(\mathcal C)$ and the \mydef{Muirhead semiring} $\mathcal M_S(\mathcal C)$ over a cone $\mathcal C$ are defined as follows:
	$\mathcal M_C(\mathcal C)$ consists of $\mathcal C$-linear combinations of
	Muirhead differences $\Set{M_\lambda-M_\mu}{\lambda\m\mu}$,
	and $\mathcal M_S(\mathcal C)$ consists of $\mathcal C$-linear combinations
	of products of functions in $\Set{M_\lambda-M_\mu}{\lambda\m\mu} \cup
	\Set{M_\lambda}{\lambda}$.
	In particular, when evaluated at $x \in [0,\infty)^n$, functions in $\mathcal M_C(\mathcal C)$ and $\mathcal M_S(\mathcal C)$ take values in the cone $\mathcal C$.
	
\begin{theorem}[{\cite[Theorem~6.1, Corollary~6.2]{CGS}}]\label{thm:majorization-muirhead}
		Given partitions $\lambda\m\mu$, the following differences
		\begin{align}
			\frac{e_{\lambda'}(x)}{e_{\lambda'}(\bm1)} - \frac{e_{\mu'}(x)}{e_{\mu'}(\bm1)},
			\quad \frac{p_\lambda(x)}{p_\lambda(\bm1)} - \frac{p_\mu(x)}{p_\mu(\bm1)}	
		\end{align} lie in the Muirhead semiring $\mathcal M_S$ over $\Q_{\geq0}$.
	\end{theorem}
	It should be noted that the Muirhead semiring $\mathcal M_S$ is strictly larger than the Muirhead cone $\mathcal M_C$; see \cref{ex:Muirhead} below. 
	
	\subsubsection{Weak majorization-type and evaluation positivity}
	
	These inequalities (specifically, the final one in the next result) were
	first recorded in \cite[Theorem~1.14]{KT} for Schur polynomials, followed
	by the other variants that were shown in~\cite{CS24}.
	
	\begin{theorem}[{\cite[Theorem 6.7]{CS24}}]\label{thm:KT}
		The following are equivalent for partitions $\lambda, \mu$:
		\begin{enumerate}
			\item $\lambda$ weakly majorizes $\mu$, $\lambda\wm\mu$.
			\item The following difference is positive:
			\begin{align}
				\frac{m_\lambda(x+\bm1)}{m_\lambda(\bm1)} - \frac{m_\mu(x+\bm1)}{m_\mu(\bm1)}\geqslant0, \quad \forall x\in[0,\infty)^n.
			\end{align}
			\item The following difference is positive:
			\begin{align}
				\frac{e_{\lambda'}(x+\bm1)}{e_{\lambda'}(\bm1)} - \frac{e_{\mu'}(x+\bm1)}{e_{\mu'}(\bm1)}\geqslant0, \quad \forall x\in[0,\infty)^n.
			\end{align}
			\item The following difference is positive:
			\begin{align}
				\frac{p_\lambda(x+\bm1)}{p_\lambda(\bm1)} - \frac{p_\mu(x+\bm1)}{p_\mu(\bm1)}\geqslant0, \quad \forall x\in[0,\infty)^n.
			\end{align}
			\item (see also \cite{KT}) The following difference is positive:
			\begin{align}
				\frac{s_\lambda(x+\bm1)}{s_\lambda(\bm1)} - \frac{s_\mu(x+\bm1)}{s_\mu(\bm1)}\geqslant0, \quad \forall x\in[0,\infty)^n.
			\end{align}
		\end{enumerate}
	\end{theorem}
	As noted in \cite{CS24}, these inequalities follow easily from their containment- and majorization- analogues, via the following easy lemma
	that connects the three partial orders.
	
	\begin{lemma} [{\cite[Lemma 6.6]{CS24}}]\label{lem:cmwm}
		$\lambda\wm\mu$ if and only if there exists some $\nu$ such that $\lambda\supseteq\nu\m\mu$.
	\end{lemma}
	
	As mentioned above, \cite{KT} provides a different proof of
	\cref{thm:KT} for Schur polynomials.

	\subsection{Our work}\label{Sourwork}
	As Jack polynomials unify Schur, monomial, and elementary symmetric polynomials, it is a pleasing fact that \cref{thm:contain-J} holds, unifying several related but disparate positive-coefficient expansion-characterizations of containment in \cref{thm:contain-J0}. In the same vein, it is natural to ask if \cref{thm:majorization,thm:KT} can also be extended to Jack polynomials for all parameters, or more strongly, to Jack-positivity in the sense of \cref{thm:contain-J}(3).

	The goal of this work is to achieve the above unifications in several cases. We also go beyond Jack polynomials, to Macdonald polynomials with parameters $q,t \in (0,1)$. We refer the reader to:
	\begin{itemize}
		\item
		\cref{thm:CGSimpliesKT,thm:CGS-KT-2var,thm:Jack2para},
		connecting (weak) majorization and normalized Jack differences;
		\item \cref{thm:Muirhead} for the Muirhead semiring and Jack polynomial
		differences;
		\item \cref{thm:CGS-Muirhead-Mac} for majorization via
		Macdonald polynomials.
	\end{itemize}

	Our results also fit into---and contribute to---the framework of \textit{duality} discussed in \cite{CS24}: given a partial order $\m_1$ on symmetric polynomials and a basis $\{ b_\lambda(x) \}$ of them, as well as a partial order $\m_2$ on partitions, we say that $\m_1$ and $\m_2$ are \mydef{dual} for the family $b_\lambda$ if the following holds:
	\begin{equation}
	b_\lambda \m_1 b_\mu \ \text{if and only if} \ \lambda \m_2 \mu.
	\end{equation}

	In \cite{CS24}, two of us showed that the containment order $\m_2$ is dual to ``Jack expansion positivity'' (wherein $p(x) \m_1 q(x)$ means $p-q$ is a positive combination of Jack polynomials---for a specific choice of parameter or for an indeterminate parameter). In this spirit, \cref{thm:majorization} asserts that majorization, or the dominance order $\m$ on partitions (which we take to be $\m_2$) is dual to ``evaluation positivity'' on the positive orthant: $p \m_1 q$ if $p(x) \geq q(x)$ for all $x \in (0,\infty)^n$---this is for normalized Schur, or monomial, or other specific families of symmetric functions. 

	Similarly, \cref{thm:KT} shows the duality between the weak dominance order $\m_2$ (which is now $\wm$) and evaluation positivity for the shifted normalized bases $f_\lambda(x + {\bf 1})$ (with the same $f$ as previously) on the positive orthant. It is our goal in this work to extend both of these dualities to Jack polynomials with arbitrary parameter---and even more strongly, to Macdonald polynomials.

	\section{Preliminaries}\label{Sprelim}
	
	In this section, we recall the definitions of the symmetric functions
	that appear in this paper, mostly following \cite{Mac15,CS24}, with minor
	modifications.
	
	\subsection{Partitions}
	
	A \mydef{partition} is an infinite tuple $\lambda = (\lambda_1,\lambda_2,\dots)$ of weakly-decreasing non-negative integers with finitely many nonzero entries.
	The non-zero $\lambda_i$'s are called the \mydef{parts} of $\lambda$.
	The zeros are usually omitted, and we use the exponential to indicate repeated entries.
	The \mydef{length} of $\lambda$, denoted by $\ell(\lambda)$, is the number of parts; and the \mydef{size} is the sum of the parts, $|\lambda|=\lambda_1+\cdots+\lambda_{\ell(\lambda)}$.
	For example, $\lambda=(2^21^3)=(2,2,1,1,1,0,\dots)$ has length 5 and size 7.
	
	The \mydef{conjugate} of a partition $\lambda$ is denoted by $\lambda'$, given by 
	\begin{align}
		\lambda_j' = \#\Set{i}{\lambda_i\geq j},\quad j\geq1.
	\end{align}
	Note that $\lambda''=\lambda$, and $\lambda_1'=\ell(\lambda)$.

	Denote by $\mathcal P_n$ the set of partitions of length at most $n$ (note, this differs from Macdonald's notation), and $\mathcal P_n'$ the set of the conjugates of $\mathcal P_n$, namely, the partitions $\lambda$ with $\lambda_1\leq n$.
	
	\subsection{Symmetric polynomials}
	
	Let $\Lambda_n$ be the algebra of symmetric polynomials in $n$ variables $x=(x_1,\dots,x_n)$ over the field $\Q$.
	The \mydef{monomial symmetric polynomial} is defined as
	\begin{align}
		m_\lambda(x) \coloneqq \sum_\eta x^\eta,
	\end{align}
	where $x^\eta\coloneqq x_1^{\eta_1}\cdots x_n^{\eta_n}$ and the sum is over distinct permutations $\eta$ of $\lambda=(\lambda_1,\dots,\lambda_n)$.
	Note that $m_\lambda$ is a well-defined function for any real $n$-tuple $\lambda$.
	
	For any $k\geq1$, the $k$-th \mydef{elementary symmetric polynomial} and \mydef{power sum} are defined by
	\begin{align}
		e_k(x) &\coloneqq m_{(1^k)}(x) = \sum_{1\leq i_1<\dots<i_k\leq n} x_{i_1}\cdots x_{i_k},\\
		p_k(x) &\coloneqq m_{(k)}(x) = \sum_{1\leq i\leq n} x_{i}^k.
	\end{align}
	(If $k>n$, $e_k(x_1,\dots,x_n) := 0$.)
	Next, define $e_\lambda$ and $p_\lambda$  by multiplication:
	\begin{align}
		e_\lambda\coloneqq e_{\lambda_1}\cdots e_{\lambda_l}, \quad p_\lambda\coloneqq p_{\lambda_1}\cdots p_{\lambda_l},   \qquad l=\ell(\lambda).
	\end{align}
	
	The \mydef{Schur polynomial} is defined by
	\begin{align}
		s_\lambda(x) = \frac{\det(x_i^{\lambda_j+n-j})_{1\leq i,j\leq n}}{\det(x_i^{n-j})_{1\leq i,j\leq n}}.
	\end{align}
	Note that the denominator is the Vandermonde determinant 
	\begin{align}
		\det(x_i^{n-j})_{1\leq i,j\leq n}=\prod_{1\leq i<j\leq n} (x_i-x_j),
	\end{align}
	and that the numerator (and hence $s_\lambda$) is a well-defined function for any real $n$-tuple $\lambda$.
	
	Each of $\Set{m_\lambda}{\lambda\in\mathcal P_n}$, $\Set{e_{\lambda'}}{\lambda\in\mathcal P_n}$, $\Set{p_{\lambda'}}{\lambda\in\mathcal P_n}$ and $\Set{s_\lambda}{\lambda\in\mathcal P_n}$ is a $\Q$-basis of $\Lambda_n$ \cite[Chapter I]{Mac15}.
	
	\subsection{Jack polynomials and Macdonald polynomials}
	
	Let $\alpha=\frac{1}{\tau}$, $q$ and $t$ be indeterminates over $\Q$ and
	consider the fields of rational functions $\F=\Q(\alpha)=\Q(\tau)$ and
	$\F=\Q(q,t)$. Let $\Lambda_{n,\tau}\coloneqq\Lambda_n\otimes\Q(\tau)$ and
	$\Lambda_{n,q,t}\coloneqq\Lambda_n\otimes\Q(q,t)$ be the algebras of
	symmetric polynomials over the larger fields.
	
	Recall the \mydef{$\tau$-Hall inner product} over $\Lambda_{n,\tau}$ and the \mydef{$q,t$-Hall inner product} over $\Lambda_{n,q,t}$:
	\begin{align}
		\langle p_\lambda,p_\mu\rangle_\tau \coloneqq \delta_{\lambda} z_\lambda \tau^{-\ell(\lambda)}, \quad
		\langle p_\lambda,p_\mu\rangle_{q,t} \coloneqq \delta_{\lambda} z_\lambda \prod_{i=1}^{\ell(\lambda)}\frac{1-q^{\lambda_i}}{1-t^{\lambda_i}},
	\end{align}
	where $\lambda,\mu\in\mathcal P_n'$, and $z_\lambda$ is a constant given in \cite[p.~24]{Mac15}.
	Then the \mydef{Jack polynomial} $P_\lambda(\tau)$ and the \mydef{Macdonald polynomial} $P_\lambda(q,t)$ can be defined as the unique polynomials satisfying:
	\begin{align}
		\langle P_\lambda(*),P_\mu(*)\rangle_{*} = 0,\quad \lambda\neq\mu,	\\
		P_\lambda(*) = \sum_{\lambda\m\mu} K_{\lambda\mu}(*)m_\mu,	\quad K_{\lambda\lambda}=1;
	\end{align}
	here $*$ is $\tau$ for Jack polynomials and $(q,t)$ for Macdonald polynomials.
	Note that our Jack polynomial $P_\lambda(x;\tau)$ is equal to Macdonald's $P_\lambda^{(\alpha=1/\tau)}(x)$.
	
	We provide additional definitions later when computing the polynomials.

	\subsection{The cone of positivity}
	
	In the Jack case, we define the \mydef{cone of positivity} $\fp$ and its real extension $\fp^{\R}$ as
	\begin{align}\label{eqn:fp}
		\fp\coloneqq \Set*{\frac fg}{f,g\in \mathbb Z_{\geq0}[\tau],\ g\neq0},\quad
		\fp^\R\coloneqq \Set*{\frac fg}{f,g\in\R_{\geq0}[\tau],\ g\neq0}.
	\end{align}
	
	In the Macdonald case, we define the \mydef{cone of positivity} $\kp$ and its real extension $\kp^\R$ as
	\begin{align}\label{eqn:fp-Mac}
		\kp\coloneqq\Set{f\in\Q(q,t)}{f(q,t)\geq0,\text{\ if\ } q,t\in(0,1)},\quad \kp^\R\coloneqq\Set{f\in\R(q,t)}{f(q,t)\geq0,\text{\ if\ } q,t\in(0,1)}.
	\end{align}

	\section{Jack polynomial conjectures and their partial resolutions}
	
	Having explained the previous literature, we now
	come to the main thrust of the present work. 
	Notice that
	\cref{thm:contain-J0} above is a ``manifestation'' of
	\cref{thm:contain-J} for special real values of the parameter $\tau$.
	Given similar manifestations of (weak) majorization in
	\cref{thm:majorization,thm:KT}, it is natural to expect the analogues of
	\cref{thm:contain-J} to also hold for (weak) majorization. 
	Thus, we begin with the following upgradation of
	\cite[Conjecture~1(1)]{CS24} to majorization. For reasons of
	exposition, we write the assertion that $\lambda\m\mu$ at the
	end:
	
	\begin{conj}[CGS Conjecture for Jack polynomials]\phantom{}\label{conj:CGS-J}
		The following are equivalent for partitions $\lambda$ and $\mu$ :
		\begin{enumerate}
			\item We have
			\begin{align}\label{eqn:CGS-J}
				\frac{P_\lambda(x;\tau)}{P_\lambda(\bm1;\tau)} -
				\frac{P_\mu(x;\tau)}{P_\mu(\bm1;\tau)}\in\fp^{\mathbb R}, \quad
				\forall x\in[0,\infty)^n.
			\end{align}
				
			\item For every $\tau_0\in[0,\infty]$, we have
			\begin{align}\label{eqn:CGS-J0}
				\frac{P_\lambda(x;\tau_0)}{P_\lambda(\bm1;\tau_0)} -
				\frac{P_\mu(x;\tau_0)}{P_\mu(\bm1;\tau_0)}\geqslant 0, \quad
				\forall x\in[0,\infty)^n.
			\end{align}
			
			\item For some $\tau_0\in[0,\infty]$, we have
			\begin{align}\label{eqn:CGS-J0log}
				\frac{P_\lambda(x;\tau_0)}{P_\lambda(\bm1;\tau_0)} -
				\frac{P_\mu(x;\tau_0)}{P_\mu(\bm1;\tau_0)}\geqslant 0, \quad
				\forall x\in(0,1)^n \cup (1,\infty)^n.
			\end{align}
			
			\item $\lambda$ majorizes $\mu$.
		\end{enumerate}
	\end{conj}
	
	\begin{remark}\label{remark:weakerhypotheses}
	We note two ways in which our conjecture (and
	results below) work with weaker hypotheses than
	\cite{CGS} for instance.
	Firstly, \cite{CGS}
	imposed the extra hypothesis $|\lambda| = |\mu|$---see e.g.\
	\cref{thm:majorization}. We do not do so, instead working with
	arbitrary partitions $\lambda, \mu$.
	 Second, the assertions in \cref{thm:majorization} all
	claim that an inequality of normalized symmetric differences
	implies majorization, when evaluated at all test vectors $x$ in
	$[0,\infty)^n$. In contrast, \cref{conj:CGS-J} (and results
	below) work only with $x \in (0,1)^n \cup (1,\infty)^n$. In
	\cref{remark:asymptote} below, we will further shrink the test
	set, to a countable set of test vectors $x$. This strengthens
	even the previously known characterizations of majorization, by
	Muirhead, Cuttler--Greene--Skandera, and others.
	\end{remark}
	
	\noindent In exact parallel, we next upgrade \cite[Conjecture~1(2)]{CS24}
	to weak majorization.
	
	\begin{conj}[KT Conjecture for Jack polynomials]\phantom{}\label{conj:CGS-Jwm}
		The following are equivalent for partitions $\lambda$ and $\mu$:
		\begin{enumerate}
			\item We have
			\begin{align}\label{eqn:KT-J}
				\frac{P_\lambda(x+\bm1;\tau)}{P_\lambda(\bm1;\tau)} -
				\frac{P_\mu(x+\bm1;\tau)}{P_\mu(\bm1;\tau)}\in\fp^{\mathbb R},
				\quad \forall x\in[0,\infty)^n.
			\end{align}
			
			\item For every $\tau_0\in[0,\infty]$, we have
			\begin{align}\label{eqn:KT-J0}
				\frac{P_\lambda(x+\bm1;\tau_0)}{P_\lambda(\bm1;\tau_0)} -
				\frac{P_\mu(x+\bm1;\tau_0)}{P_\mu(\bm1;\tau_0)}\geqslant 0, \quad
				\forall x\in[0,\infty)^n.
			\end{align}
			
			\item For some $\tau_0\in[0,\infty]$, we have
			\begin{align}\label{eqn:KT-J1}
				\frac{P_\lambda(x+\bm1;\tau_0)}{P_\lambda(\bm1;\tau_0)} -
				\frac{P_\mu(x+\bm1;\tau_0)}{P_\mu(\bm1;\tau_0)}\geqslant 0, \quad
				\forall x\in[0,\infty)^n.
			\end{align}
			
			\item $\lambda$ weakly majorizes $\mu$.
		\end{enumerate}
	\end{conj}

	In the remaining parts of this section, we prove the following partial results:

	\begin{theorem}\label{thm:CGS-KT-2var}
		\cref{conj:CGS-J,conj:CGS-Jwm} hold for Jack polynomials in $n=2$ variables, as well as for any $n$ and $\mu = {\bf 1}_{|\lambda|}$.
		Moreover, the assertions (2)--(4) are equivalent in either conjecture whenever $\lambda$ and $\mu$ have at most two parts (with any number of variables).
	\end{theorem}
	\begin{remark}
		An analogue of the implication $(4)\Rightarrow(2)$ in \cref{conj:CGS-J} was formulated by McSwiggen and Novak for Heckman--Opdam polynomials for arbitrary root systems, and proved in semisimple rank 1, subsuming the case of Jack polynomials in two variables. 
		See \cite[Conjecture~4.7, Proposition~4.10]{MN}.		
		Although we only work in type~$A$, our results for two variables, as well as our \cref{conj:CGS-J,conj:CGS-Jwm} and later \cref{conj:Muirhead,conj:CGS-Muirhead-Mac,conj:jack-2para}, go considerably beyond~\cite{MN}.
	\end{remark}

	\subsection{The reduction to $(4) \Rightarrow (1)$}

	We first show ``most of'' the implications in \cref{conj:CGS-J,conj:CGS-Jwm}, reducing them to only one unproved implication.
	\begin{theorem}\label{thm:CGSimpliesKT}
		In both \cref{conj:CGS-J,conj:CGS-Jwm}, each part implies the next.
		Moreover, the remaining cyclic implication $(4)\Rightarrow(1)$ in
		\cref{conj:CGS-Jwm} follows from $(4)\Rightarrow(1)$ in \cref{conj:CGS-J}.
	\end{theorem}
		\begin{proof}
		$\bullet$
		In both \cref{conj:CGS-J,conj:CGS-Jwm}, the implications $(1)\Rightarrow(2)\Rightarrow(3)$ are evident.
		
		$\bullet$
		We now prove $(3)\Rightarrow(4)$ in both conjectures.
		
		First, assume that $\lambda$ does not weakly majorize $\mu$, namely, $d_k(\lambda) \coloneqq \lambda_1+\cdots+\lambda_k<\mu_1+\cdots+\mu_k= d_k(\mu)$ for some $k$.
		Let 
		\begin{align*}
			x(t) = (\underbrace{2t,\dots,2t}_{\text{$k$ times}},\underbrace{2,\dots,2}_{\text{$n-k$ times}})\in(1,\infty)^n,	\quad t>1.
		\end{align*}		
		Then as polynomials in $t$, $\frac{P_\lambda(x(t);\tau_0)}{P_\lambda(\bm1;\tau_0)}$ and $\frac{P_\mu(x(t);\tau_0)}{P_\mu(\bm1;\tau_0)}$ have degree $d_k(\lambda)$ and $d_k(\mu)$, respectively, since the leading term (in the majorization order) of $P_\lambda$ is $m_\lambda$.
		Now, let $t\to\infty$, we see that $$\frac{P_\lambda(x(t);\tau_0)}{P_\lambda(\bm1;\tau_0)}-\frac{P_\mu(x(t);\tau_0)}{P_\mu(\bm1;\tau_0)}\to-\infty,$$
		contradicting \cref{eqn:CGS-J0log,eqn:KT-J1} (in \cref{eqn:KT-J1}, we view the argument $x(t)$ as $(x(t)-\bm1)+\bm1$, with $x(t)-\bm1\in[0,\infty)^n$).
		
		Second, assume that $-\lambda \coloneqq (-\lambda_n,\dots,-\lambda_1)$ does not weakly majorize $-\mu=(-\mu_n,\dots,-\mu_1)$, namely, $s_k(\lambda)\coloneqq \lambda_n+\dots+\lambda_{n-k+1}>\mu_n+\dots+\mu_{n-k+1}=s_k(\mu)$.
		Let 
		\[
		y(t)=\Big(\underbrace{\tfrac1{2t},\dots,\tfrac1{2t}}_{k\text{ times}},
		\underbrace{\tfrac12,\dots,\tfrac12}_{n-k\text{ times}}\Big)\in(0,1)^n,
		\qquad t>1.
		\]
		
		Let $\nu$ be a partition.
		Then
		\[
		m_\nu(y(t))
		=\Big(\tfrac12\Big)^{|\nu|}\sum_{\beta\in\mathfrak S_n\cdot \nu} t^{-(\beta_1+\cdots+\beta_k)}.
		\]
		The minimum of $\beta_1+\cdots+\beta_k$ over all permutations $\beta$ of $\nu$ is attained by placing
		the $k$ smallest parts of $\nu$ into the first $k$ coordinates, so this minimum equals $s_k(\nu)=\nu_{n-k+1}+\cdots+\nu_n$.
		Therefore there exists a constant $A_{\nu,k}>0$ such that
		\begin{equation}\label{eq:mon-asymp}
			m_\nu(y(t))
			=A_{\nu,k}\,t^{-s_k(\nu)}+O(t^{-s_k(\nu)-1}),
			\qquad t\to\infty.
		\end{equation}
		
		Note that Jack polynomials expand positively into monomials:
		\[
		P_\lambda(x;\tau_0)= m_\lambda(x) + \sum_{\nu\prec\lambda} c_{\lambda\nu} m_\nu(x),
		\qquad c_{\lambda\nu}\geq0.
		\]
		For $\nu\prec\lambda$, we have 
		\[
		s_k(\nu) = \nu_n+\dots+\nu_{n-k+1}\geq \lambda_n+\dots+\lambda_{n-k+1}=s_k(\lambda).
		\]
		Hence there exists a constant $A_{\lambda,k}'>0$ such that 
		\begin{equation}\label{eq:Pl-asymp}
			P_\lambda(y(t);\tau_0)= A_{\lambda,k}' t^{-s_k(\lambda)} +O\big(t^{-s_k(\lambda-1}\big),
			\qquad t\to\infty,
		\end{equation}
		and similarly for $P_\lambda(y(t);\tau_0)$.
		Then we have
		\[
		\frac{P_\lambda(y(t);\tau_0)/P_\lambda(\bm1;\tau_0)}{P_\mu(y(t);\tau_0)/P_\mu(\bm1;\tau_0)}
		\sim \frac{A_{\lambda,k}'/P_\lambda(\bm1;\tau_0)}{A_{\mu,k}'/P_\mu(\bm1;\tau_0)}\,t^{-(s_k(\lambda)-s_k(\mu))} \longrightarrow 0,
		\qquad t\to\infty.
		\]
		Thus for all sufficiently large $t$,
		\[
			\frac{P_\lambda(y(t);\tau_0)}{P_\lambda(\bm1;\tau_0)} -
			\frac{P_\mu(y(t);\tau_0)}{P_\mu(\bm1;\tau_0)}<0,
		\]
		contradicting \cref{eqn:CGS-J0log}, hence $-\lambda$ weakly dominates $-\mu$.
		
		To summarize, we see that \cref{eqn:KT-J1} implies $\lambda\wm\mu$, and that \cref{eqn:CGS-J0log} implies $\lambda\wm\mu$ and $-\lambda\wm-\mu$, forcing $|\lambda|=|\mu|$ and hence $\lambda\m\mu$.
		These are exactly the implications $(3)\Rightarrow(4)$ in the conjectures.
		
		$\bullet$
		Finally, we prove that \cref{conj:CGS-Jwm} follows from \cref{conj:CGS-J}.
		
		Assume \cref{conj:CGS-Jwm}(4), namely, $\lambda$ weakly majorizes $\mu$. 
		Then by \cref{lem:cmwm}, there exists $\nu$ such
		that $\lambda\supseteq\nu\m\mu$. Now write (suppressing the $\tau$)
		\begin{align*}
			\frac{P_\lambda(x+\bm1)}{P_\lambda(\bm1)} - \frac{P_\mu(x+\bm1)}{P_\mu(\bm1)}
			=	
			\left(\frac{P_\lambda(x+\bm1)}{P_\lambda(\bm1)} -\frac{P_\nu(x+\bm1)}{P_\nu(\bm1)}\right) +\left(\frac{P_\nu(x+\bm1)}{P_\nu(\bm1)} - \frac{P_\mu(x+\bm1)}{P_\mu(\bm1)}\right).
		\end{align*}
		The first difference is Jack-positive by \cref{thm:contain-J},
		and in particular, in $\fp^{\mathbb R}$ when evaluated at $x\in[0,\infty)^n$. 
		The second difference is in $\fp^{\mathbb R}$ by \cref{eqn:CGS-J} (for the pair $\nu\m\mu$).
	\end{proof}
	
	\begin{remark}
		As a consequence of \cref{thm:CGSimpliesKT}, to completely settle \cref{conj:CGS-J,conj:CGS-Jwm}, it suffices to show that
		$\lambda\m\mu$ $\Rightarrow$ \cref{eqn:CGS-J}. Note that for
		$\tau_0=0,1,\infty$, this is already known by \cref{thm:majorization}. 
	\end{remark}

	\begin{remark}\label{remark:asymptote}
		As mentioned in \cref{remark:weakerhypotheses}, we now reduce the test set $(0,1)^n \cup (1,\infty)^n$ to a potentially ``minimum'', countable test set. Let $S\subset \Set{x\in\R^n}{x_1>\dots>x_n>0}$ be any subset.
		We say that $S$ has asymptotes in direction $i$ ($1\leq i\leq n-1$), if there exists a constant $M>0$ and a sequence $(x^{(1)},x^{(2)},\dots)$ in $S$, such that 
		\begin{align*}
		\begin{dcases}
			\lim_{m\to\infty}\frac{x_i^{(m)}}{x_{i+1}^{(m)}}\to \infty,\\
			\lim_{m\to\infty}\frac{x_j^{(m)}}{x_{j+1}^{(m)}}<M,\quad	j=1,\dots,i-1,i+1,\dots,n-1,\\
			\lim_{m\to\infty} x_n^{(m)}<M.
		\end{dcases}
		\end{align*}
		We say that $S$ has asymptotes in direction $n$, if there exists a constant $M>0$ and a sequence $(x^{(1)},x^{(2)},\dots)$ in $S$, such that 
		\begin{align*}
			\begin{dcases}
				\lim_{m\to\infty}\frac{x_j^{(m)}}{x_{j+1}^{(m)}}<M,\quad	j=1,\dots,n-1,\\
				\lim_{m\to\infty} x_n^{(m)}=0,\text{ or }\infty.
			\end{dcases}
		\end{align*}
		
		To prove \cref{eqn:CGS-J0} for some fixed $\tau_0$, one may assume without lost of generality that $x_1>\dots>x_n>0$ by the symmetry of Jack polynomials. We now have:
		
		\begin{proposition}
		Let $S\subset \Set{x\in\R^n}{x_1>\dots>x_n>0}$ be a subset that has asymptotes in all directions $1,\dots,n$. If
		\begin{align}\label{eqn:CGS-J-S}
			\frac{P_\lambda(x;\tau_0)}{P_\lambda(\bm1;\tau_0)} -
			\frac{P_\mu(x;\tau_0)}{P_\mu(\bm1;\tau_0)}\geqslant 0, \quad
			\forall x\in S,
		\end{align}
		then $\lambda\m\mu$.
		\end{proposition}
		
		\begin{proof}
		For direction $k=1,\dots,n-1$ and for $k=n$ with $\lim_{m\to\infty}x_n^{(m)}=\infty$, we see that $x_1^{(m)},\dots,x_k^{(m)}$ approach infinity at the same rate and $x_{k+1}^{(m)},\dots,x_n^{(m)}$ are bounded.
		Then the ``order'' of $P_\lambda(x^{(m)};\tau_0)$ is $\lambda_1+\cdots+\lambda_k$, and similarly for $P_\mu(x^{(m)};\tau_0)$.
		Since \cref{eqn:CGS-J-S} holds for all $x\in S$, letting $m\to\infty$, we see that $\lambda_1+\cdots+\lambda_k\geq\mu_1+\cdots+\mu_k$.
		As this holds for all $k$, we have $\lambda\wm\mu$.
		
		For direction $k=n$ with $\lim_{m\to\infty}x_n^{(m)}=0$, we see that $x_1^{(m)},\dots,x_n^{(m)}$ approach 0 at the same rate, and the ``order'' of $P_\lambda(x;\tau_0)$ is $-|\lambda|$.
		Since \cref{eqn:CGS-J-S} holds for all $x\in S$, we see that $|\lambda|\leq |\mu|$.
		Hence we conclude that $|\lambda|=|\mu|$ and $\lambda\m\mu$.
	\end{proof}	
	\end{remark}

	\subsection{Reductions of the conjectures}
	In this subsection, we further reduce \cref{conj:CGS-J} to some important special cases.
	
	\subsubsection{Adding columns}
	\begin{proposition}\label{prop:one_row}
		For a one-row partition $(d)=(d,0^{n-1})$, the Jack polynomials are given by
		\begin{align}
			P_{(d,0^{n-1})}(x;\tau) = \frac{d!}{\poch{\tau}{d}} \cdot \sum_{|\eta|=d} \frac{\poch{\tau}{\eta}}{\eta!}x^\eta,
		\end{align}
		where the sum runs over all $n$-tuples $\eta=(\eta_1,\dots,\eta_n)\in\Z_{\geq0}^n$ such that $\eta_1+\dots+\eta_n=d$, and 
		\begin{align}
			\poch{\tau}{\eta}	\coloneqq	\prod_i \poch{\tau}{\eta_i},\quad	\poch{x}{n}\coloneqq x(x+1)\cdots(x+n-1) = \frac{\Gamma(x+n)}{\Gamma(x)}
		\end{align}
		is the rising factorial.
		When evaluated at $\bm1$, we have
		$\displaystyle P_{(d,0^{n-1})}(\bm1; \tau) = \frac{\poch{n\tau}{d}}{\poch{\tau}{d}}$.
	\end{proposition}
	\begin{proof}
		The first assertion follows from the combinatorial formula (see
		e.g.~\cite[Eq.~(2.10)]{CS24}). The second follows from
		\cite[VI.~(10.20)]{Mac15} (note, $\alpha$ is $\frac1\tau$ here),
		or the Chu--Vandermonde identity.
	\end{proof}
	
	\begin{proposition}\label{prop:add_col}
		Let $\lambda\in\mathcal P_n$ be such that $\lambda_n>0$. Write $\lambda-\bm1=(\lambda_1-1,\dots,\lambda_n-1)$.
		Then 
		\begin{align}
			P_\lambda(x; \tau) = x_1\cdots x_n \cdot P_{\lambda-\bm1}(x; \tau).
		\end{align}
	\end{proposition}
	\begin{proof}
		See \cite[VI.~(4.17)]{Mac15} for the Macdonald case.
	\end{proof}
	
	\begin{lemma}[Adding columns]\label{lem:add_col}
		To prove \cref{conj:CGS-J} (i.e., that $\lambda\m\mu$ $\Rightarrow$
		\cref{eqn:CGS-J}), it suffices to assume $\lambda_n=0$.
	\end{lemma}
	
	\begin{proof}
		Let $\lambda\m\mu$, then $\lambda_n\leq\mu_n$. 
		Clearly, $\lambda-\lambda_n\bm1\m\mu-\lambda_n\bm1$, and by assumption, \cref{eqn:CGS-J} holds for the pair $(\lambda-\lambda_n\bm1,\mu-\lambda_n\bm1)$. 
		By \cref{prop:add_col}, \cref{eqn:CGS-J} holds for the pair $(\lambda,\mu)$ as well.
	\end{proof}
	\subsubsection{Raising operator}
	\begin{lemma}[Raising operator]\label{lem:raisingop}
		To prove \cref{conj:CGS-J} ($\lambda\m\mu\Rightarrow$ \cref{eqn:CGS-J}), it suffices to assume that $\lambda$ and $\mu$ are adjacent in the majorization order.
		In that case, $\lambda=R_{ij}(\mu)$ for some $i<j$, where
		$R_{ij}(\mu)=(\dots,\mu_i+1,\dots,\mu_{j}-1,\dots)$ is the
		raising operator, see \cite[pp.~8]{Mac15}.
	\end{lemma}
	
	\begin{proof}
		The first claim follows by transitivity.
		The second claim is \cite[I.~(1.16)]{Mac15}.
	\end{proof}
	
	Note that if $\lambda$ and $\mu$ are majorization-adjacent, then they are of exactly two possible forms:
	\begin{itemize}
		\item $\lambda=(\nu,k+l,l,\eta)$ and $\mu=(\nu,k+l-1,l+1,\eta)$, where $k\geq2$, $l\geq0$, and $\nu$ and $\eta$ (possibly empty) are partitions that make $\lambda$ and $\mu$ valid partitions;
		
		\item $\lambda=(\nu,k+1,k^m,k-1,\eta)$ and $\mu=(\nu,k^{m+2},\eta)$, where $k,m\geq1$, and $\nu$ and $\eta$ (possibly empty) are partitions that make $\lambda$ and $\mu$ valid partitions.
	\end{itemize}
	\subsubsection{Appending 0}
	We first recall the integration formula due to Okounkov--Olshanski \cite{OO97}.
	
	Set $\mathbb R^n_{\geq}\coloneqq\Set{x\in\mathbb R^n}{x_1\geq\cdots\geq
		x_n}$. For $x\in\mathbb R^n_{\geq}$ and $y\in\mathbb R^{n-1}_{\geq}$,
	write $y\prec_{\mathrm{int}} x$ (interlacing) if 
	\begin{align}
		x_1\geq y_1 \geq x_2 \geq y_2 \geq \cdots \geq x_{n-1} \geq y_{n-1} \geq x_n.
	\end{align}
	
	Let $\lambda\in\mathcal P_{n-1}$ be a partition of length at most $n-1$, and $(\lambda,0)\in\mathcal P_n$.
	Then \cite[pp.~77]{OO97} (where $\theta$ is our $\tau$) shows that for
	$x\in\mathbb R^n_{\geq}$,
	\begin{align}
		P_{(\lambda,0)}(x) = \frac{1}{C(\lambda,n; \tau)V(x)^{2\tau-1}} \int_{y\prec_{\mathrm{int}}x} P_\lambda(y) V(y) \Pi(x,y)^{\tau-1} \dy,
	\end{align}
	where 
	\begin{align*}
		C(\lambda,n; \tau)	&\coloneqq	\prod_{i=1}^{n-1} \mathop{\mathrm{B}}(\lambda_i+(n-i)\tau,\tau),	\\
		V(x)	&\coloneqq	\prod_{i<j} (x_i-x_j),	\\
		\Pi(x,y)	&\coloneqq	\prod_{i=1}^n\prod_{j=1}^{n-1}|x_i-y_j|	=	\prod_{i\leq j}(x_i-y_j) \prod_{i>j}(y_j-x_i),
	\end{align*}
	and $\mathop{\mathrm{B}}(\alpha,\beta)$ is the Beta function.
	
	Okounkov--Olshanski's formula has a better form when changed to the unital normalization.
	\begin{proposition}
		For $x\in\mathbb R^n_\geq$ and $y\in\mathbb R^{n-1}_\geq$, define the integral kernel 
		$K(x,y)$ by
		\begin{equation}\label{eqn:Kxy}
			K(x,y) \coloneqq K_n(x,y;\tau) = \frac{V(y)\Pi(x,y)^{\tau-1}}{V(x)^{2\tau-1}} \frac{\Gamma(n\tau)}{\Gamma(\tau)^n}.
		\end{equation}
		Then
		\begin{align}
			\frac{P_{(\lambda,0)}(x)}{P_{(\lambda,0)}(\bm1_n)} = \int_{y\prec_{\mathrm{int}} x} \frac{P_{\lambda}(y)}{P_{\lambda}(\bm1_{n-1})} K(x,y) \dy, \qquad \forall \lambda\in\mathcal P_{n-1}, \ x\in\mathbb R^n_\geq.
		\end{align}
		In particular, $\int_{y\prec_{\mathrm{int}} x} K(x,y) \dy=1$.
	\end{proposition}
	
	\begin{proof}
		By \cite[VI.~(10.20)]{Mac15}, we have for any $\mu\in\mathcal P_n$,
		\begin{align*}
			P_\mu(\bm1_{n}) = \prod_{s\in\mu}\frac{a_\mu'(s)+(n-l_\mu'(s))\tau}{a_\mu(s)+(l_\mu(s)+1)\tau}.
		\end{align*}
		Now, 
		\begin{align*}
			\frac{P_{(\lambda,0)}(\bm1_{n}) }{P_\lambda(\bm1_{n-1})} 
			=	\prod_{s\in\lambda}\frac{a_\lambda'(s)+(n-l_\lambda'(s))\tau}{a_\lambda'(s)+(n-1-l_\lambda'(s))\tau} 	
			=	\prod_{i=1}^{n-1} \frac{\poch{(n-i+1)\tau}{\lambda_i}}{\poch{(n-i)\tau}{\lambda_i}}
			=	\prod_{i=1}^{n-1} \frac{\frac{\Gamma(\lambda_i+(n-i+1)\tau)}{\Gamma((n-i+1)\tau)}}{\frac{\Gamma(\lambda_i+(n-i)\tau)}{\Gamma((n-i)\tau)}}
		\end{align*}
		and hence
		\begin{align*}
			\frac{P_{(\lambda,0)}(\bm1_{n}) }{P_\lambda(\bm1_{n-1})} \frac{1}{C(\lambda,n; \tau)} 
			=	\prod_{i=1}^{n-1} \frac{\Gamma(\tau)\Gamma((n-i)\tau)}{\Gamma((n-i+1)\tau)}	
			=	\frac{\Gamma(\tau)^n}{\Gamma(n\tau)}.	&\qedhere
		\end{align*}
	\end{proof}
	
	\begin{lemma}[Appending 0]\label{lem:appending0}
		If \cref{eqn:CGS-J0} holds for the pair $(\lambda,\mu)$ in $n-1$ variables, then \cref{eqn:CGS-J0} holds for the pair $((\lambda,0),(\mu,0))$ in $n$ variables as well.
	\end{lemma}
	\begin{proof}
		Assume that \cref{eqn:CGS-J} holds for the pair $(\lambda,\mu)$ in $n-1$ variables $y=(y_1,\dots,y_{n-1})$.
		When $\tau=0$, Jack polynomials reduce to monomial symmetric polynomials and \cref{eqn:CGS-J0} is simply Muirhead's inequality.
		For every $\tau>0$, $x\in\R_{\geq0}^n$, since the kernel $K(x,y)>0$, we have 
		\begin{align*}
			\frac{P_{(\lambda,0)}(x)}{P_{(\lambda,0)}(\bm1_n)}  - \frac{P_{(\mu,0)}(x)}{P_{(\mu,0)}(\bm1_n)}
			= \int_{y\prec_{\mathrm{int}} x} \left(\frac{P_{\lambda}(y)}{P_{\lambda}(\bm1_{n-1})}-\frac{P_{\mu}(y)}{P_{\mu}(\bm1_{n-1})}\right) K(x,y) \dy>0.
		\end{align*}
		Since the LHS is symmetric in $x$, \cref{eqn:CGS-J0} holds for all $x\in[0,\infty]^n$. 
	\end{proof}
	
	\subsection{An example and the proof}
	\begin{example}\label{ex:CGS}
		Consider $n=2$, $d\geq2$, $\lambda=(d,0)$ and $\mu=(d-1,1)$.
		Suppressing the parameter $\tau$, we have
		\begin{align}
			\frac{P_\lambda(x_1,x_2)}{P_\lambda(1,1)} 
			= \sum_{i=0}^d \binom{d}{i} \frac{\poch{\tau}{i}\poch{\tau}{d-i}}{\poch{2\tau}{d}} x_1^{d-i} x_2^{i}
		\end{align}
		and
		\begin{align}
			\frac{P_\mu(x_1,x_2)}{P_\mu(1,1)} 
			&=	\frac{x_1x_2}{1}\frac{P_{(d-2,0)}(x_1,x_2)}{P_{(d-2,0)}(1,1)} 
			=	\sum_{i=0}^{d-2} \binom{d-2}{i} \frac{\poch{\tau}{i}\poch{\tau}{d-i-2}}{\poch{2\tau}{d-2}} x_1^{d-i-1}x_2^{i+1}.
		\end{align}
		Setting $x=(1,s)$, $s>0$, we have
		\begin{align*}
			f(s)	&\coloneqq	\frac{P_\lambda(1,s)}{P_\lambda(1,1)} - \frac{P_\mu(1,s)}{P_\mu(1,1)}	\\
			&=	\sum_{i=0}^d \binom{d}{i} \frac{\poch{\tau}{i}\poch{\tau}{d-i}}{\poch{2\tau}{d}} s^i 	-\sum_{i=0}^{d-2} \binom{d-2}{i} \frac{\poch{\tau}{i}\poch{\tau}{d-i-2}}{\poch{2\tau}{d-2}} s^{i+1}	\\
			&=	\sum_{i=0}^{d} \left(\binom{d}{i} \frac{\poch{\tau}{i}\poch{\tau}{d-i}}{\poch{2\tau}{d}} - \binom{d-2}{i-1} \frac{\poch{\tau}{i-1}\poch{\tau}{d-i-1}}{\poch{2\tau}{d-2}}\right) s^{i}	
			\eqqcolon \sum_i f_i s^i.
		\end{align*}
		Define $g(s)\coloneqq \sum_{i=0}^{d-2} g_is^i$, where 
		\begin{align*}
			g_i \coloneqq \binom{d-2}{i} \frac{(\tau +d-1)}{\tau} \frac{\poch{\tau}{i+1}\poch{\tau}{d-i-1}}{\poch{2\tau}{d}} \in\fp,\quad 0\leq i\leq d-2.
			\label{eqn:g_i}
		\end{align*}
		Let $g_{-1}=g_{-2}=g_{d-1}=g_d=0$. 
		Then one can show that 
		\begin{align*}
			f_i = g_{i-2}-2g_{i-1}+g_i,
		\end{align*}
		hence
		\begin{align*}
			f(s) = (s-1)^2g(s)\in\fp^{\mathbb R},\quad s>0.
		\end{align*}
		Now, \cref{eqn:CGS-J} follows since
		\begin{align*}
			\frac{P_\lambda(x_1,x_2)}{P_\lambda(1,1)} - \frac{P_\mu(x_1,x_2)}{P_\mu(1,1)} = x_2^d \cdot f(x_1/x_2)\in\fp^{\mathbb R},\quad (x_1,x_2)\in(0,\infty)^n.
		\end{align*}
	\end{example}
	\begin{proof}[Proof of \cref{thm:CGS-KT-2var}]
		In order to prove \cref{conj:CGS-J,conj:CGS-Jwm} in the case of two variables, it suffices to prove $\lambda\m\mu\Rightarrow$ \cref{eqn:CGS-J} by \cref{thm:CGSimpliesKT}.
		By \cref{lem:raisingop}, it suffices to show this implication for the majorization-adjacent pair: $\lambda=(d+l,l)\m\mu=(d+l-1,l+1)$, $l\geq0$, $d\geq2$. 
		By \cref{lem:add_col}, it can be further reduced to $\lambda=(d,0)\m\mu=(d-1,1)$, and this case was proved in \cref{ex:CGS}.
		
		The next assertion -- in which $\mu = {\bf 1}_{|\lambda|}$ -- is proved in \cref{thm:MuirheadSemiring,thm:Muirhead} and the remark in between them.

		We next show the final assertion. By \cref{lem:appending0}, the weaker implication $(4)\Rightarrow(2)$ in \cref{conj:CGS-J} holds for partitions $\lambda$ and $\mu$ of length at most 2 and any number of variables. 
		We are now done by \cref{thm:CGSimpliesKT}.
	\end{proof}

	\subsection{Supporting evidence: Kadell's integral}
	
	We end this section by providing supporting evidence for the
	assertion $(4) \Rightarrow (2)$ in \cref{conj:CGS-J}. We begin by recalling
	Kadell's integral in \cite[VI.\ Ex.~10.7]{Mac15} and \cite{Kad97} (with a different
	normalization). Namely, let $\lambda$ be a partition of length at most
	$n$, as usual, and let $r,s>0$ (or more generally, $r,s\in\mathbb C$ with
	$\Re r,\Re s>0$). Then
	\begin{align}
		I_\lambda=I_\lambda(\tau;r,s)&\coloneqq	\int_{[0,1]^n} \frac{P_\lambda(x;\tau)}{P_\lambda(\bm1;\tau)} |V(x)|^{2\tau} \prod_{i=1}^n x_i^{r-1}(1-x_i)^{s-1} \dx \notag
		\\&=	\prod_{i=1}^n \frac{\Gamma((n-i)\tau+\lambda_{i}+r) \Gamma((i-1)\tau+s) \Gamma(i\tau+1)}{\Gamma((2n-i-1)\tau+\lambda_{i}+r+s)\Gamma(\tau+1)},
	\end{align}
	where $x=(x_1,\dots,x_n)$, $\dx=\dx_1\cdots\dx_n$, and $V(x)=\prod_{i<j} (x_i-x_j)$.
	
	Let $\tilde I_\lambda\coloneqq I_\lambda/I_{(0^n)}$ be the normalized
	Kadell's integral. Then $\tilde I_{(0^n)}=1$ and
	\begin{align}
		\tilde I_\lambda=\tilde I_\lambda(\tau;r,s)
		&=	\prod_{i=1}^n \frac{\Gamma((n-i)\tau+\lambda_{i}+r)}{\Gamma((n-i)\tau+r)} \frac{\Gamma((2n-i-1)\tau+r+s)}{\Gamma((2n-i-1)\tau+\lambda_{i}+r+s)} \notag
		\\&=	\prod_{i=1}^n \frac{\poch{(n-i)\tau+r}{\lambda_i}}{\poch{(2n-i-1)\tau+r+s}{\lambda_i}}
	\end{align}
	is a rational polynomial in $\tau$.
	We now show:
	\begin{proposition}
		If $\lambda\m\mu$, then $\tilde I_\lambda-\tilde I_\mu\in\fp^\R$. 
		If moreover $r$ and $s$ are rational numbers, then $\tilde I_\lambda-\tilde I_\mu\in\fp$.
	\end{proposition}
	\begin{proof}
		Note that $\tilde I_\mu\in\fp^\R$ for $r,s>0$.
		By \cref{lem:raisingop}, we may assume there exist $1\leq i<j\leq n$ such that $\lambda=R_{ij}(\mu)$, i.e., $\lambda_i=\mu_i+1$, $\lambda_j=\mu_j-1$ and $\lambda_k=\mu_k$ for $k\neq i,j$.
		Since $\lambda_i-\lambda_j-1\geq0$,
		\begin{align*}
			\tilde I_\lambda/\tilde I_\mu -1
			&=	\frac{(n-i)\tau+\lambda_i+r-1}{(n-j)\tau+\lambda_j+r}
			\cdot \frac{(2n-j-1)\tau+\lambda_j+r+s}{(2n-i-1)\tau+\lambda_i+r+s}-1
			\\&=	\frac{((j-i)\tau+\lambda_i-\lambda_j-1)((n-1)\tau+s)}{((n-j)\tau+\lambda_j+r)((2n-i-1)\tau+\lambda_i+r+s-1)}\in\fp^\R. \qedhere
		\end{align*}
	\end{proof}

	This supports $(4) \Rightarrow (2)$ in \cref{conj:CGS-J}, since
	for each fixed $\tau_0 \in (0,\infty)$, if the difference
	$\frac{P_\lambda(x;\tau_0)}{P_\lambda(\bm1;\tau_0)}-\frac{P_\mu(x;\tau_0)}{P_\mu(\bm1;\tau_0)}$
	is positive over $[0,1]^n$, then so is its integral against Kadell's
	density 
	\begin{align*}
		\frac{1}{I_{(0^n)}} |V(x)|^{2\tau_0} \prod_{i=1}^n x_i^{r-1}(1-x_i)^{s-1}>0.
	\end{align*}

	\section{Jack polynomial differences and the Muirhead semiring}
	
	In this section we prove \cref{conj:CGS-J} (i.e., that $\lambda\m\mu
	\Rightarrow$ \cref{eqn:CGS-J}) in some more cases. In fact, we prove a
	stronger version of the conjecture, in which evaluation positivity is replaced
	by sums and products of Muirhead differences (which are positive by Muirhead's
	inequality).
	
	\subsection{Jack differences in the Muirhead semiring}
	Recall the Muirhead cone and semiring, defined above \cref{thm:majorization-muirhead}.
	
	In \cite[Section 7]{CGS}, the authors wondered if normalized Schur differences  $\frac{s_\lambda(x)}{s_\lambda(\bm1)}-\frac{s_\mu(x)}{s_\mu(\bm1)}$ lie in the Muirhead cone, for $\lambda\m\mu$.
	However, this is not the case as the example below shows.
	It seems that the result would hold if one replaces the Muirhead cone by the semiring---which is strictly larger, as the example also confirms.

	\begin{example}\label{ex:Muirhead}
		Consider $n=3$, $\lambda=(5,2,0)$, and $\mu=(5,1,1)$.
		We have
		\begin{align*}
			70\cdot 
			\Bigg(\frac{s_\lambda(x)}{s_\lambda(\bm{1})}-\frac{s_\mu(x)}{s_\mu(\bm{1})}\Bigg)
			=
			10 M_{(5,2,0)}
			-9 M_{(5,1,1)}
			+10 M_{(4,3,0)}
			-8 M_{(4,2,1)}
			-4 M_{(3,3,1)}
			+M_{(3,2,2)}.
		\end{align*}
		
		Note that the coefficient of $M_{(3,2,2)}$ is positive and $(3,2,2)$ is the minimum (in the majorization order) partition that appears in the expansion of $\frac{s_\lambda(x)}{s_\lambda(\bm{1})}-\frac{s_\mu(x)}{s_\mu(\bm{1})}$, and hence the difference cannot lie in the Muirhead cone.
		
		However, this difference indeed lies in the Muirhead semiring, in fact for Jack polynomials with arbitrary parameter $\tau = \tau_0$. Even more strongly, this holds for the cone $\mathcal{C} = \mathbb{F}_{\geqslant 0}$:
		\begin{align*}
			&\=	
			3(2\tau+1)(3\tau+1)(3\tau+2)(3\tau+4)
			\Bigg(\frac{P_\lambda(x;\tau)}{P_\lambda(\bm{1};\tau)}-\frac{P_\mu(x;\tau)}{P_\mu(\bm{1};\tau)}\Bigg)\\
			&=	(2\tau^{4}+19\tau^{3}+43\tau^{2}+54\tau+24) (M_{(5,2,0)}-M_{(5,1,1)})\\
			&\=	+\tau^{2}(2\tau^{2}+7\tau+7) (M_{(5,2,0)}-M_{(4,3,0)})\\
			&\=	+\tau(14\tau^{3}+61\tau^{2}+75\tau+36) (M_{(4,3,0)}-M_{(4,2,1)})\\
			&\=	+\tau^{3}(10\tau+23) (M_{(4,2,1)}-M_{(3,3,1)})\\
			&\=	+\tau^{3}(10\tau+11) (M_{(3,3,1)}-M_{(3,2,2)})\\
			&\= +15\tau^{2} M_{(1,0,0)}(M_{(3,3,0)}-M_{(3,2,1)})\\
			&\=	6\tau(6\tau+5) (M_{(2,0,0)}-M_{(1,1,0)})(M_{(3,2,0)}-M_{(3,1,1)})\\
			&\=	+6\tau (M_{(2,0,0)}-M_{(1,1,0)})(M_{(3,1,1)}-M_{(2,2,1)}),
		\end{align*}
		which is in the Muirhead semiring $\mathcal M_S(\fp)$.
	\end{example}
		
	Motivated by this calculation, we make the following conjecture. 
	\begin{conj}\label{conj:Muirhead}
		The following are equivalent for partitions $\lambda$ and $\mu$:
		\begin{enumerate}
			\item The Jack difference lies in the Muirhead semiring
			over $\fp$:
			\begin{equation}
				\frac{P_\lambda(x;\tau)}{P_\lambda(\bm1;\tau)} - \frac{P_\mu(x;\tau)}{P_\mu(\bm1;\tau)}\in \mathcal M_S(\fp).
			\end{equation}
			\item For every $\tau_0\in[0,\infty]$, the Jack difference lies in the Muirhead semiring
			over $\R_{\geq0}$:
			\begin{equation}
				\frac{P_\lambda(x;\tau_0)}{P_\lambda(\bm1;\tau_0)} - \frac{P_\mu(x;\tau_0)}{P_\mu(\bm1;\tau_0)}\in \mathcal M_S(\R_{\geq0}).
			\end{equation}
			\item For some $\tau_0\in[0,\infty]$, the Jack difference lies in the Muirhead semiring
			over $\R_{\geq0}$:
			\begin{equation}\label{eqn:Muirhead-conj}
				\frac{P_\lambda(x;\tau_0)}{P_\lambda(\bm1;\tau_0)} - \frac{P_\mu(x;\tau_0)}{P_\mu(\bm1;\tau_0)}\in \mathcal M_S(\R_{\geq0}).
			\end{equation}
			\item $\lambda$ majorizes $\mu$.
		\end{enumerate}
	\end{conj}	

	\begin{theorem}\label{thm:MuirheadSemiring}
		In \cref{conj:Muirhead}, clearly $(1)\Rightarrow(2)\Rightarrow(3)$, and (3) implies \cref{conj:CGS-J}(3), which gives $\lambda\m\mu$.
		So \cref{conj:Muirhead} reduces to showing $(4)\Rightarrow(1)$.
		\qed
	\end{theorem}

	\begin{remark}
		By Muirhead's inequality, \cref{conj:Muirhead}(1) implies \cref{conj:CGS-J}(1).
		Thus \cref{conj:Muirhead} implies \cref{conj:CGS-J}.
	\end{remark}

	We now show
	\begin{theorem}\label{thm:Muirhead}
		\cref{conj:CGS-J,conj:CGS-Jwm,conj:Muirhead} hold for all $\lambda$
		when $\mu=(1^{|\lambda|})$. 
		\cref{conj:Muirhead} moreover holds in two variables.
	\end{theorem}
	\begin{proof}
		If $\lambda \m \mu = (1^{|\lambda|})$, we prove more strongly that the difference lies in the Muirhead cone $\mathcal M_C(\fp)$.
		Write
		\[
		\frac{P_\lambda}{P_\lambda(\bm1)}=\sum_{\nu\preceq\lambda} u_{\lambda\nu} M_\nu,\quad u_{\lambda\nu}\in\fp.
		\]
		Then $\sum_\nu u_{\lambda\nu}=1$ by our choice of normalization. 
		Also, we have $\frac{P_\mu}{P_\mu(\bm1)}=M_\mu$.
		Hence,
		\begin{equation}\label{eqn:Muirhead-cone}
			\frac{P_\lambda}{P_\lambda(\bm1)} - \frac{P_\mu}{P_\mu(\bm1)} = \sum_\nu u_{\lambda\nu} (M_\nu-M_\mu) \in\mathcal M_C(\fp).
		\end{equation}
		This shows \cref{conj:Muirhead}, and hence
		\cref{conj:CGS-J,conj:CGS-Jwm} for $\mu = (1^{|\lambda|})$.
		
		We next come to the case of two variables. Here we
		continue the calculations in \cref{ex:CGS}. Expanding the
		difference of the normalized Jack polynomials in Muirhead differences
		and suppressing the parameter $\tau$ as usual, we obtain:
		\begin{align*}
			\frac{P_\lambda(x)}{P_\lambda(\bm1)} - \frac{P_\mu(x)}{P_\mu(\bm1)} 
			=\sum_{i=0}^{d}
			\left(\binom{d}{i}\frac{\poch{\tau}{i}\poch{\tau}{d-i}}{\poch{2\tau}{d}}
			-
			\binom{d-2}{i-1}\frac{\poch{\tau}{i-1}\poch{\tau}{d-i-1}}{\poch{2\tau}{d-2}}\right)
			x_1^{d-i}x_2^{i}.
		\end{align*}
		Let 
		\begin{align*}
			a_i &\coloneqq 2
			\left(\binom{d}{i}\frac{\poch{\tau}{i}\poch{\tau}{d-i}}{\poch{2\tau}{d}}
			-
			\binom{d-2}{i-1}\frac{\poch{\tau}{i-1}\poch{\tau}{d-i-1}}{\poch{2\tau}{d-2}}\right),
			\quad 0\leq i\leq \lfloor\frac{d-1}{2}\rfloor,\\
			a_{d/2} &\coloneqq
			\binom{d}{d/2}\frac{\poch{\tau}{d/2}\poch{\tau}{d/2}}{\poch{2\tau}{d}}
			-
			\binom{d-2}{d/2-1}\frac{\poch{\tau}{d/2-1}\poch{\tau}{d/2-1}}{\poch{2\tau}{d-2}},
			\text{\quad if $d$ even}.
		\end{align*}
		Then by Abel's lemma,
		\begin{align*}
			\frac{P_\lambda(x)}{P_\lambda(\bm1)} - \frac{P_\mu(x)}{P_\mu(\bm1)} 
			= \sum_{i=0}^{\lfloor\frac{d}{2}\rfloor} a_i M_{(d-i,i)}(x) = \sum_{i=0}^{\lfloor\frac{d}{2}\rfloor-1} \left(\sum_{j\leq i} a_j\right) (M_{(d-i,i)}(x)-M_{(d-i-1,i+1)}(x))
		\end{align*}
		where the summand $M_{(d/2,d/2)}$ is only present when $d$ is even.
		It is not hard to prove by telescoping that
		\begin{align*}
			\sum_{j\leq i} a_j = 2\binom{d-1}{i} \frac{d-1-2 i}{d-1} \frac{\poch{\tau}{i} \poch{\tau}{d-i-1}}{\poch{2\tau}{d}} (\tau+d-1)\in\fp,\quad i=0,\dots,\lfloor\frac{d}{2}\rfloor-1.
		\end{align*}
		Note that if $d$ is even, then $\sum_{j\leq d/2} a_j =0$.
	\end{proof}

	We end this part with a two-variable consequence (in fact reformulation)
	of the final part of \cref{thm:Muirhead}.
	
	\begin{corollary}\label{cor:Muirhead}
		Let $x=(x_1,x_2)$ and fix $\lambda\m\mu$ partitions with at most two parts.
		We have 
		\[
		\frac{1}{V(x)^2} \left( \frac{P_\lambda(x; \tau)}{P_\lambda(\bm1; \tau)} -
		\frac{P_\mu(x; \tau)}{P_\mu(\bm1; \tau)} \right) \in
		\bigoplus_\nu \mathbb{F}_{\geq 0} \cdot s_\nu(x).
		\]
	\end{corollary}
	
	\begin{remark}
		Notice that \cref{cor:Muirhead} does not hold for higher $n$---not even
		for $n=3$, even at $\tau = 1$, i.e.\ for Schur polynomials. 
		(This is in parallel to  \cref{conj:Muirhead} also not
		holding for the Muirhead \textit{cone} even with $n=3$ and $\tau=1$, as
		verified in \cref{ex:Muirhead}.)
		For example, working with $(3) \m (1,1,1)$ yields
		\[
		s_{(3)}(x_1, x_2, x_3) - s_{(3)}(1,1,1) s_{(1,1,1)}(x_1, x_2, x_3) =
		\sum_i x_i^3 + \sum_{i<j} (x_i^2 x_j + x_j^2 x_i) - 9 x_1 x_2 x_3,
		\]
		and this is not divisible by $x_i - x_j$ for any $i \neq j$. (To see why,
		set $x_i = x_j$ in the above expression, and note that the expression
		does not vanish.) Nor is it monomial-positive.
	\end{remark}
	
	\begin{proof}[Proof of \cref{cor:Muirhead}]
		In fact for two variables, not only does the conjecture imply this
		assertion, but the converse is also true. This is because---we claim
		that---every normalized Muirhead-difference $M_\lambda(x_1,x_2) -
		M_\mu(x_1,x_2)$ (for $\lambda\m\mu$), when divided by $V(x)^2 =
		(x_1-x_2)^2$, is a Schur-positive polynomial with half-integer
		coefficients.
		This claim reduces by transitivity to majorization-adjacent partitions
		$\lambda = (a+b+1,a), \mu = (a+b,a+1)$ (for $a \geq 0, b \geq 1$), for
		which it holds by an explicit computation:
		\[
		\frac{2}{(x_1-x_2)^2} \left( M_\lambda(x_1,x_2) - M_\mu(x_1,x_2) \right)
		= s_{(a+b-1,a)}(x_1,x_2). \qedhere
		\]
	\end{proof}

	\subsection{Jack polynomial differences with varying parameter}
	
	Since $m_\lambda(x)=P_\lambda(x;\tau=0)$, \cref{eqn:Muirhead-conj} can be viewed as a positivity property involving Jack polynomial differences in two different parameters $\tau$ and $0$.
	Let us generalize this idea further.
	
	Define the \mydef{Jack cone} and \mydef{Jack semiring} over the cone $\fp$ as follows:
	$\mathcal J_C^\sigma(\R_{\geq0})$ consists of $\R_{\geq0}$-linear combinations of
	Jack differences $\Set{\frac{P_\lambda(x;\sigma)}{P_\lambda(\bm1;\sigma)}-\frac{P_\mu(x;\sigma)}{P_\mu(\bm1;\sigma)}}{\lambda\m\mu}$,
	and $\mathcal J_S^\sigma(\R_{\geq0})$ consists of $\R_{\geq0}$-linear combinations
	of products of functions in $\Set{\frac{P_\lambda(x;\sigma)}{P_\lambda(\bm1;\sigma)}-\frac{P_\mu(x;\sigma)}{P_\mu(\bm1;\sigma)}}{\lambda\m\mu} \cup
	\Set{\frac{P_\lambda(x;\sigma)}{P_\lambda(\bm1;\sigma)}}{\lambda}$.
	
	\begin{conj}\label{conj:jack-2para}
		Fix $0\leq\sigma<\tau\leq\infty$. If $\lambda\m\mu$, then the Jack
		difference with parameter $\tau$ is positive in the Jack semiring $\mathcal J_S^\sigma(\R_{\geq0})$:
		\begin{align}\label{eqn:Jack-vary-para}
			\frac{P_\lambda(x;\tau)}{P_\lambda(\bm1;\tau)} - \frac{P_\mu(x;\tau)}{P_\mu(\bm1;\tau)}\in \mathcal J_S^\sigma(\R_{\geq0}).
		\end{align}
	\end{conj}
	
	The following example shows that \cref{conj:jack-2para} is false when $\tau<\sigma$. 
	\begin{example}
		Consider $n=3$, $\tau=0$, $\sigma=1$, and let
		\[
		F\coloneqq3(M_{(3)}-M_{(21)}) = 10S_{(3)}-12S_{(21)}+2S_{(111)},
		\]
		where $S_\lambda=s_\lambda/s_\lambda(\bm1)$.
		Clearly, it does not belong to the Schur cone.
		We claim that $F$ does not belong to the Schur semiring.

		Note that in degree 3, the Schur semiring is generated by
		\begin{align*}
			f_1 &\coloneqq S_{(1)}^3, \quad
			f_2 \coloneqq S_{(1)} S_{(11)}, \quad
			f_3 \coloneqq S_{(111)},\\
			f_4 &\coloneqq S_{(1)}\bigl(S_{(2)}-S_{(11)}\bigr), \quad
			f_5 \coloneqq S_{(21)}-S_{(111)}, \quad
			f_6 \coloneqq S_{(3)}-S_{(21)} .
		\end{align*}
		It is not hard to see that $f_3,f_5,f_6$ are enough to generate the semiring since		
		\begin{align*}
			f_1 = f_3 + \frac{26}{27}\,f_5 + \frac{10}{27}\,f_6, \quad
			f_2 = f_3 + \frac{8}{9}\,f_5,\quad
			f_4 = \frac{1}{9}\,f_5 + \frac{5}{9}\,f_6.
		\end{align*}
		However, $F=10f_6-2f_5$, hence $F$ does not lie in the Schur semiring.
	\end{example}
	
	Now, we again look at the two-variable case in \cref{ex:CGS}. This time, the computation is more complicated.
	\begin{example}\label{ex:2para_n=2}
		Consider $n=2$, $d\geq2$, $\lambda=(d,0)$ and $\mu=(d-1,1)$.
		In this example, we mostly suppress the argument $x$, and also assume the indices $i,j=0,\dots,\lfloor d/2 \rfloor$ unless otherwise stated.
		
		We prove the stronger statement the difference in \cref{eqn:Jack-vary-para} lies in the Jack cone $\mathcal J_C^\sigma(\R_{\geq0})$. 
		
		The transition matrix from $\left(\frac{P_\lambda(x;\tau)}{P_\lambda(\bm1;\tau)}\right)$ to $(P_\lambda(x;\tau))$ is a diagonal matrix
		\begin{equation}\label{eqn:transition}
			M_d\left(\frac{P(\tau)}{P(\bm1;\tau)},P(\tau)\right)\coloneqq \left(\frac1{P_{(d-i,i)}(\bm1;\tau)}\delta_{ij}\right)_{i,j} = \left(\frac{\poch{\tau}{d-2i}}{\poch{2\tau}{d-i}}\delta_{ij}\right)_{i,j}.
		\end{equation}
		
		By \cref{prop:one_row,prop:add_col}, we have
		\begin{align*}
			P_{(d-i,i)}(x; \tau) =\sum_j \binom{d-2i}{j-i}\frac{\poch{\tau}{j-i}\poch{\tau}{d-j-i}}{\poch{\tau}{d-2i}} \cdot m_{(d-j,j)}(x),
		\end{align*}
		hence the transition matrix from $\left(P_\lambda(x;\tau)\right)$ to $\left(m_\lambda(x)\right)$ is 
		\begin{align*}
			M_d(P(\tau),m)\coloneqq
			\left(	\binom{d-2i}{j-i}\frac{\poch{\tau}{j-i}\poch{\tau}{d-j-i}}{\poch{\tau}{d-2i}}	\right)_{i,j}.
		\end{align*}
		For example, when $d=2,3$, we have 
		\begin{align*}
			M_2(P(\tau),m) = \begin{pmatrix}
				1&\frac{2\tau}{1+\tau}\\&1
			\end{pmatrix},\quad
			M_3(P(\tau),m) = \begin{pmatrix}
				1&\frac{3\tau}{1+\tau}\\&1
			\end{pmatrix},
		\end{align*}
		that is, $P_{(2,0)}(\tau) = m_{(2,0)}+\frac{2\tau}{1+\tau}m_{(1,1)}$
		and $P_{(3,0)}(\tau) = m_{(3,0)}+\frac{3\tau}{1+\tau}m_{(2,1)}$.
		
		Define a matrix 
		\begin{align}
			N_d(\sigma) = \left(	(-1)^{i-j}\binom{d-i-j}{j-i}\frac{d-2i}{d-i-j}\frac{\poch{\sigma+i-j+1}{d-i-j}\poch{\sigma+d-i-j+1}{j-i}}{\poch{\sigma+1}{d-2i}}	\right)_{i,j}.
			\label{eqn:M-sigma-m-inv}
		\end{align}
		(When $i=j=d/2$, the factor $\frac{d-2i}{d-i-j}$ is understood to be 1.)
		For example, when $d=2,3$, we have
		\begin{align*}
			N_d(\sigma) = \begin{pmatrix}
				1&-\frac{2\sigma}{1+\sigma}\\&1
			\end{pmatrix},\quad
			N_d(\sigma) = \begin{pmatrix}
				1&-\frac{3\sigma}{1+\sigma}\\&1
			\end{pmatrix}.
		\end{align*}
		
		We claim that the product matrix $M_d\left(P(\tau),m\right)N_d(\sigma)$ is given by \begin{align}\label{eqn:M-tau-sigma}
			M_d\left(P(\tau),m\right) N_d(\sigma)
			&=	\left(	\frac{(d-2i)!}{(d-2j)!(j-i)!} \frac{\poch{\tau-\sigma}{j-i}}{\poch{\tau+d-i-j}{j-i}\poch{\sigma+d-2j+1}{j-i}}	\right)_{i,j}.
		\end{align}
		(When $i>j$, the entries are understood to be 0.)
		In particular, when setting $\sigma=\tau$, we see that the product $M_d\left(P(\tau),m\right) N_d(\tau)$ is the identity matrix, and hence $N_d(\tau)=M_d\left(P(\tau),m\right)^{-1}$ is the inverse transition matrix, and $M_d\left(P(\tau),m\right) N_d(\sigma) =M_d\left(P(\tau),P(\sigma)\right)$ is the transition matrix between Jack polynomials with different parameters.
		
		Given \cref{eqn:M-sigma-m-inv}, to prove \cref{eqn:M-tau-sigma}, it suffices to check the strictly upper diagonal entries.
		Let $i<j$, and compute the $(i,j)$-entry of the product $M_d\left(P(\tau),m\right)M_d\left(P(\sigma),m\right)^{-1}$:
		\begin{align*}
			&\= \sum_{k=i}^j \binom{d-2i}{k-i}\frac{\poch{\tau}{k-i}\poch{\tau}{d-k-i}}{\poch{\tau}{d-2i}} (-1)^{k-j} \times
			\\&\=\qquad	\binom{d-k-j}{j-k}\frac{d-2k}{d-k-j}\frac{\poch{\sigma+k-j+1}{d-k-j}\poch{\sigma+d-k-j+1}{j-k}}{\poch{\sigma+1}{d-2k}}
			\\&= \frac{(d-2i)!}{(d-2j)!(j-i)!}\frac{1}{\poch{\tau}{d-2i}} \cdot
			\sum_{k=i}^j (-1)^{k-j} \binom{j-i}{k-i} \frac{(d-k-j)!}{(d-k-i)!} \frac{d-2k}{d-k-j} \times
			\\&\=\qquad	\poch{\tau}{k-i}\poch{\tau}{d-k-i} \frac{\poch{\sigma+k-j+1}{d-k-j}\poch{\sigma+d-k-j+1}{j-k}}{\poch{\sigma+1}{d-2k}}
			\\&= \frac{(d-2i)!}{(d-2j)!(j-i)!}\frac{1}{\poch{\tau}{d-2i}} (-1)^{i-j}\cdot \sum_{l=0}^{j-i} (-1)^{l} \binom{j-i}{l} \frac{(d-l-i-j)!}{(d-l-2i)!} \frac{d-2l-2i}{d-l-i-j} \times
			\\&\=\qquad	\poch{\tau}{l}\poch{\tau}{d-l-2i} \frac{\poch{\sigma+l+i-j+1}{d-l-i-j}\poch{\sigma+d-l-i-j+1}{j-i-l}}{\poch{\sigma+1}{d-2l-2i}}
		\end{align*}
		where in the last identity, we set $l=k-i$.
		
		In other words, it suffices to show that 
		\begin{align}\label{eqn:WZ}
			& \sum_{l=0}^{j-i} (-1)^{l} \binom{j-i}{l} \frac{(d-l-i-j)!}{(d-l-2i)!} \frac{d-2l-2i}{d-l-i-j}\times
			\notag\\ &\qquad \poch{\tau}{l}\poch{\tau}{d-l-2i} \frac{\poch{\sigma+l+i-j+1}{d-l-i-j}\poch{\sigma+d-l-i-j+1}{j-i-l}}{\poch{\sigma+1}{d-2l-2i}}
			\\= & (-1)^{i-j} \frac{\poch{\tau}{d-2i}\poch{\tau-\sigma}{j-i}}{\poch{\tau+d-i-j}{j-i}\poch{\sigma+d-2j+1}{j-i}}.\notag
		\end{align}
		Denote by $f(i,j,d;l)$ the summand on the left, and $r(i,j,d)$ the RHS.
		We prove \cref{eqn:WZ} using the Wilf--Zeilberger algorithm \cite[Chapter 7]{a=b}.
		Let 
		\begin{align*}
			&\=	F(i,j,d;l)	\coloneqq	f(i,j,d;l)/r(i,j,d)
			\notag\\&=	(-1)^{j-i-l} \binom{j-i}{l} \frac{(d-l-i-j)!}{(d-l-2i)!} \frac{d-2l-2i}{d-l-i-j} \cdot \frac{1}{\poch{\tau-\sigma}{j-i}}
			\frac{\poch{\tau}{l}\poch{\tau}{d-l-2i}\poch{\tau+d-i-j}{j-i}}{\poch{\tau}{d-2i}} \times
			\notag\\&\=\qquad	\frac{\poch{\sigma+l+i-j+1}{d-l-i-j}\poch{\sigma+d-l-i-j+1}{j-i-l}\poch{\sigma+d-2j+1}{j-i}}{\poch{\sigma+1}{d-2l-2i}}
			\\&=	(-1)^{j-i-l} \binom{j-i}{l} \frac{(d-l-i-j)!}{(d-l-2i)!} \frac{d-2l-2i}{d-l-i-j} \cdot \frac{1}{\poch{\tau-\sigma}{j-i}}
			\frac{\poch{\tau}{l}\poch{\tau}{d-l-2i}\poch{\tau+d-i-j}{j-i}}{\poch{\tau}{d-2i}} \times
			\notag\\&\=\qquad	\frac{\poch{\sigma+l+i-j+1}{d-l-2i}\poch{\sigma+d-l-i-j+1}{j-i-l}}{\poch{\sigma+1}{d-2l-2i}},
		\end{align*}
		and let 
		\begin{align*}
			G(i,j,d;l)\coloneqq F(i,j,d;l)  
			\cdot \frac{l(\tau+d-l-2i)(\sigma-2j-1+d)(\sigma+l+i-j)}{(d-2l-2i)(-j+i+l-1)(\sigma-j+d-i)(-\tau+\sigma-j+i)}.
		\end{align*}

		It can be checked that
		\begin{align*}
			F(i,j+1,d;l)-F(i,j,d;l) = G(i,j,d;l+1)-G(i,j,d;l).
		\end{align*}

		Hence,
		\begin{align*}
			\sum_{l\geq0} (F(i,j+1,d;l)-F(i,j,d;l)) 
			&=	\sum_{l\geq0} (G(i,j,d;l+1)-G(i,j,d;l)) 
			\notag\\&=	(G(i,j,d;j-i+1)-G(i,j,d;0))	=0.
		\end{align*}
		In other words, we have 
		\begin{align*}
			\sum_{l\geq0} F(i,j+1,d;l)=\sum_{l\geq0} F(i,j,d;l).
		\end{align*}

		Note that, when $j=i+1$, we have 
		\begin{align*}
			\sum_{l\geq0} F(i,i+1,d;l) = -\frac{(d-2i+\tau-1)\sigma}{(d-1-2i)(\tau-\sigma)} + \frac{\tau(d-1-2i+\sigma)}{(d-1-2i)(\tau-\sigma)} = 1,
		\end{align*}
		and so for any $j>i$,
		$\sum_{l=0}^{j-i} F(i,j,d;l) =\sum_{l\geq0} F(i,j,d;l) = 1$.
		We have thus proved \cref{eqn:M-tau-sigma}. 
		
		Next, we show (as asserted just before \cref{eqn:transition} that the difference in \cref{eqn:Jack-vary-para} belongs to the Jack cone $\mathcal J_C^\sigma(\R_{\geq0})$. The transition matrix from $\left(\frac{P_\lambda(\tau)}{P_\lambda(\bm1;\tau)}\right)$ to $\left(\frac{P_\lambda(\sigma)}{P_\lambda(\bm1;\sigma)}\right)$ is
		\begin{align*}
			\notag&\=	M_d\left(\frac{P(\tau)}{P(\bm1;\tau)},\frac{P(\sigma)}{P(\bm1;\sigma)}\right)
			\\\notag&=
			M_d\left(\frac{P(\tau)}{P(\bm1;\tau)},P(\tau)\right)
			M_d\left(P(\tau),P(\sigma)\right) 
			M_d\left(\frac{P(\sigma)}{P(\bm1;\sigma)},P(\sigma)\right)^{-1}
			\\&=	\left(	\frac{\poch{\tau}{d-2i}}{\poch{2\tau}{d-2i}}
			\cdot \frac{(d-2i)!}{(d-2j)!(j-i)!} \frac{\poch{\tau-\sigma}{j-i}}{\poch{\tau+d-i-j}{j-i}\poch{\sigma+d-2j+1}{j-i}} 
			\cdot \frac{\poch{2\sigma}{d-2j}}{\poch{\sigma}{d-2j}}	\right).
		\end{align*}
		
		In particular, we are interested in the difference of the first two rows, that is, 
		the coefficient of $\displaystyle \frac{P_{(d-j,j)}(\sigma)}{P_{(d-j,j)}(\bm1;\sigma)}$ in $\displaystyle\frac{P_{(d,0)}(\tau)}{P_{(d,0)}(\bm1;\tau)} - \frac{P_{(d-1,1)}(\tau)}{P_{(d-1,1)}(\bm1;\tau)}$.
		For $j=0$, this coefficient is evidently positive. 
		For $1\leq j\leq \lfloor d/2\rfloor$, this coefficient is
		\begin{align*} 
			\notag&\=	\frac{\poch{\tau}{d}}{\poch{2\tau}{d}} \cdot
			\frac{d!}{(d-2j)!j!} \frac{\poch{\tau-\sigma}{j}}{\poch{\tau+d-j}{j}\poch{\sigma+d-2j+1}{j}} \cdot\frac{\poch{2\sigma}{d-2j}}{\poch{\sigma}{d-2j}}
			\\\notag&\=	- \frac{\poch{\tau}{d-2}}{\poch{2\tau}{d-2}} \cdot
			\frac{(d-2)!}{(d-2j)!(j-1)!} \frac{\poch{\tau-\sigma}{j-1}}{\poch{\tau+d-1-j}{j-1}\poch{\sigma+d-2j+1}{j-1}} \cdot\frac{\poch{2\sigma}{d-2j}}{\poch{\sigma}{d-2j}}
			\\\notag&=	\frac{\poch{\tau}{d-2}}{\poch{2\tau}{d-2}} \cdot
			\frac{(d-2i)!}{(d-2j)!(j-i)!} \frac{\poch{\tau-\sigma}{j-1}}{\poch{\tau+d-1-j}{j-1}\poch{\sigma+d-2j+1}{j-1}} \cdot\frac{\poch{2\sigma}{d-2j}}{\poch{\sigma}{d-2j}}
			\\\notag&\=	\cdot\left(\frac{(\tau+d-2)(\tau+d-1)}{(2\tau+d-2)(2\tau+d-1)} \frac{d(d-1)}{j} \frac{\tau-\sigma+j-1}{\frac{(\tau+d-2)(\tau+d-1)}{\tau+d-j-1} (\sigma+d-j)} -1\right)
			\\\notag&\propto	\frac{d(d-1)}{j} \frac{\tau-\sigma+j-1}{(2\tau+d-2)(2\tau+d-1)} \frac{\tau+d-j-1}{\sigma+d-j} -1
			\\&\propto	2(2\tau-1) j^2 -2 (2\tau-1)(d+\sigma)j + d(d-1)(\tau-\sigma-1)\eqqcolon a_j,
		\end{align*}
		where $\propto$ means that a factor in $\fp$ has been omitted.
		
		Note that the last entry, $a_{\lfloor d/2\rfloor}$, is always negative: for $d=2k$ even, we have $a_k=-2k(\tau+k-1)(2\sigma+1)<0$; for $d=2k+1$ odd, we have $a_k=-2k(\tau+k)(2\sigma+1)<0$.
		
		For $j=1,\dots,\lfloor d/2\rfloor-1$, we have $a_{j+1}-a_j =-2(2\tau-1)(\sigma+d-2j-1)$, hence $(a_j)$ is monotone. 
		
		Now we have the expansion
		\begin{align*}
			\frac{P_{(d,0)}(\tau)}{P_{(d,0)}(\bm1;\tau)} - \frac{P_{(d-1,1)}(\tau)}{P_{(d-1,1)}(\bm1;\tau)}
			= \frac{P_{(d,0)}(\bm1;\sigma)}{P_{(d,0)}(\bm1;\tau)}\frac{P_{(d-j,j)}(\sigma)}{P_{(d-j,j)}(\bm1;\sigma)} + \sum_{j=1}^{\lfloor d/2\rfloor} c_ja_j \frac{P_{(d-j,j)}(\sigma)}{P_{(d-j,j)}(\bm1;\sigma)},
		\end{align*}
		where $c_j>0$ denotes the factor omitted above, and $(a_j)$ is monotone and ends at a negative number. 
		In particular, there exists some $j_0$ such that $a_j\geq 0$ for $j<j_0$ and $a_j\leq 0$ for $j\geq j_0$.
		Such an expansion lies in the Jack cone by Abel's lemma.
	\end{example}
	
	This calculation leads to:
	
	\begin{theorem}\label{thm:Jack2para}
		\cref{conj:jack-2para} holds when $n=2$.
		It also holds for the pair $\lambda=(2,1^{n-2},0)\m\mu=(1^n)$, when $n\geq3$.
	\end{theorem}

	\begin{proof}
		Again, we prove the stronger statement the difference in \cref{eqn:Jack-vary-para} lies in the Jack cone $\mathcal J_C^\sigma(\R_{\geq0})$.
		First, let $n=2$.
		The example above shows that \cref{conj:jack-2para} holds for the pair $\lambda=(d,0)\m\mu=(d-1,1)$.
		It suffices to show that it holds for any adjacent pair $\lambda=(d-i,i)\m\mu=(d-i-1,i+1)$, as the general case would follow by telescoping. 
		This holds because
		\begin{align*}
			\frac{P_{(d-i,i)}(x;\tau)}{P_{(d-i,i)}(\bm1;\tau)}-\frac{P_{(d-i-1,i+1)}(x;\tau)}{P_{(d-i-1,i+1;\tau)}(\bm1;\tau)}
			&= (x_1x_2)^{i} \cdot \left(\frac{P_{(d-2i,0)}(x;\tau)}{P_{(d-2i,0)}(\bm1;\tau)}-\frac{P_{(d-2i-1,1)}(x;\tau)}{P_{(d-2i-1,1)}(\bm1;\tau)}\right)	
			\\&\in	(x_1x_2)^{i} \cdot 
			\sum_{\nu\m\xi} \R_{\geq0}\cdot \left(\frac{P_\nu(x;\sigma)}{P_\nu(\bm1;\sigma)}-\frac{P_\xi(x;\sigma)}{P_\xi(\bm1;\sigma)}\right)
			\\&=	\sum_{\nu+(i,i)\m\xi+(i,i)} \R_{\geq0}\cdot \left(\frac{P_{\nu+(i,i)}(x;\sigma)}{P_{\nu+(i,i)}(\bm1;\sigma)}-\frac{P_{\xi+(i,i)}(x;\sigma)}{P_{\xi+(i,i)}(\bm1;\sigma)}\right)
			\subset	\mathcal J_C(\R_{\geq0}).
		\end{align*}
		
		Now, let $n\geq3$, $\lambda=(2,1^{n-2},0)$, and $\mu=(1^n)$.
		We have 
		\begin{align*}
			P_\lambda(x; \tau) = m_\lambda(x) + \frac{n(n-1)\tau}{1+(n-1)\tau} m_{\mu}(x),\quad 	P_\mu(x; \tau) &= m_{\mu}(x).
		\end{align*}
		Note that $m_\lambda(\bm1) = n(n-1)$ and $m_\mu(\bm1)=1$, and so 
		\begin{align*}
			\frac{P_\lambda(x;\tau)}{P_\lambda(\bm1;\tau)} = \frac{1+(n-1)\tau}{1+n\tau} M_\lambda(x) + \frac{\tau}{1+n\tau} M_{\mu}(x),\quad
			\frac{P_\mu(x;\tau)}{P_\mu(\bm1;\tau)} = M_{\mu}(x).
		\end{align*}
		Thus,
		\begin{align}\label{eqn:JU2MU}
			\frac{P_\lambda(x;\tau)}{P_\lambda(\bm1;\tau)} - \frac{P_\mu(x;\tau)}{P_\mu(\bm1;\tau)} = \frac{1+(n-1)\tau}{1+n\tau}(M_\lambda(x)-M_\mu(x)).
		\end{align}
		
		Inverting \cref{eqn:JU2MU}, we have 
		\begin{align*}
			M_\lambda(x)-M_\mu(x) = \frac{1+n\sigma}{1+(n-1)\sigma} \left(\frac{P_\lambda(x;\sigma)}{P_\lambda(\bm1;\sigma)} - \frac{P_\mu(x;\sigma)}{P_\mu(\bm1;\sigma)}\right).
		\end{align*}
		Hence,
		\begin{align*}
			\frac{P_\lambda(x;\tau)}{P_\lambda(\bm1;\tau)} - \frac{P_\mu(x;\tau)}{P_\mu(\bm1;\tau)}
			= \frac{1+(n-1)\tau}{1+n\tau}\frac{1+n\sigma}{1+(n-1)\sigma} \left(\frac{P_\lambda(x;\sigma)}{P_\lambda(\bm1;\sigma)} - \frac{P_\mu(x;\sigma)}{P_\mu(\bm1;\sigma)}\right).
		\end{align*}
		Hence in this case, $\frac{P_\lambda(x;\tau)}{P_\lambda(\bm1;\tau)} - \frac{P_\mu(x;\tau)}{P_\mu(\bm1;\tau)}$ is a positive multiple of $\frac{P_\lambda(x;\sigma)}{P_\lambda(\bm1;\sigma)} - \frac{P_\mu(x;\sigma)}{P_\mu(\bm1;\sigma)}$ if $\tau, \sigma \in[0,\infty]$.
	\end{proof}
	
	\section{Macdonald polynomial differences: conjectures and results}
	
	Let $P_\lambda(x)=P_\lambda(x;q,t)$ denote the monic Macdonald polynomials.
	Our goal in this section is to explore if \cref{conj:Muirhead}, characterizing majorization in terms of (normalized) Jack differences belonging to the Muirhead semiring, can be strengthened to Macdonald polynomials---leading to a strengthening of \cref{conj:CGS-J} as well. And indeed, we achieve both of these goals, in the cases where we have shown the conjectures above.
	
	We begin by setting notation.
	Recall that the $q$-Pochhammer, $q$-integer, $q$-factorial, and $q$-binomial are defined respectively as follows:
	\begin{gather}
		\qpoch{a}{q}{k} \coloneqq \prod_{i=0}^{k-1} (1-aq^i), 
		\quad 
		\qnum{k}\coloneqq \sum_{i=0}^{k-1} q^i = \begin{dcases}
			\frac{q^k-1}{q-1},&q\neq1;\\
			k,&q=1,
		\end{dcases}
		\quad
		\qnum{k}! \coloneqq \prod_{i=1}^k \qnum{i},	\\
		\qbinom{m}{n} \coloneqq \frac{\qnum{m}!}{\qnum{n}!\qnum{m-n}!}=\frac{\qpoch{q}{q}{m}}{\qpoch{q}{q}{n}\qpoch{q}{q}{n-m}},\quad k\geq0, m\geq n\geq0.
	\end{gather}

	We now study Macdonald differences for $\lambda \m \mu$, starting with the case of two variables. The first question in extending the aforementioned conjectures is: how should one normalize the Macdonald polynomials. We begin by exploring the natural choice \cite{Mac15} of $x = t^\delta=(t^{n-1},\dots,t,1)$. 
	
	\begin{example}
		Let $n=2$, $\lambda=(d,0)$ and $\mu=(d-1,1)$. We have, using the combinatorial formula \cite[VI.~(7.13$'$)]{Mac15},
		\begin{align}
			P_\lambda(x) = \sum_{i=0}^d \frac{\qpoch{q}{q}{d}}{\qpoch{q}{q}{i}\qpoch{q}{q}{d-i}} \frac{\qpoch{t}{q}{i}\qpoch{t}{q}{d-i}}{\qpoch{t}{q}{d}} x_1^i x_2^{d-i} 
			=	\frac{\qpoch{q}{q}{d}}{\qpoch{t}{q}{d}} \sum_{i=0}^d \frac{\qpoch{t}{q}{i}}{\qpoch{q}{q}{i}} \frac{\qpoch{t}{q}{d-i}}{\qpoch{q}{q}{d-i}} x_1^i x_2^{d-i}.\label{eqn:Mac-d0}
		\end{align}
		By \cite[VI.~(6.11)]{Mac15}, where $n=2$ and $t^\delta=(t,1)$,
		$\displaystyle P_\lambda(t^\delta) = \frac{\qpoch{t^2}{q}{d}}{\qpoch{t}{q}{d}}$.
		Hence, 
		\begin{align*}
			\frac{P_\lambda(x)}{P_\lambda(t^\delta)} = \frac{\qpoch{q}{q}{d}}{\qpoch{t^2}{q}{d}} \sum_{i=0}^d \frac{\qpoch{t}{q}{i}}{\qpoch{q}{q}{i}} \frac{\qpoch{t}{q}{d-i}}{\qpoch{q}{q}{d-i}} x_1^i x_2^{d-i}.
		\end{align*}
		
		For $\mu=(d-1,1)$, we have $P_\mu(x) = x_1x_2 P_{(d-2,0)}(x)$, hence
		\begin{align*}
			\frac{P_\mu(x)}{P_\mu(t^\delta)} 
			&=	\frac{x_1x_2}{t}\frac{\qpoch{q}{q}{d-2}}{\qpoch{t^2}{q}{d-2}} \sum_{i=0}^{d-2} \frac{\qpoch{t}{q}{i}}{\qpoch{q}{q}{i}} \frac{\qpoch{t}{q}{d-2-i}}{\qpoch{q}{q}{d-2-i}} x_1^i x_2^{d-2-i}
			\\&=	\frac{\qpoch{q}{q}{d-2}}{t\qpoch{t^2}{q}{d-2}} \sum_{i=1}^{d-1} \frac{\qpoch{t}{q}{i-1}}{\qpoch{q}{q}{i-1}} \frac{\qpoch{t}{q}{d-1-i}}{\qpoch{q}{q}{d-1-i}} x_1^{i} x_2^{d-i}.
		\end{align*}

		For example, let $d=2$. Then
		$\displaystyle \frac{P_{(1,1)}(x)}{P_{(1,1)}(t^\delta)} = \frac{1}{t}x_1x_2$, and
		\begin{align*}
			\frac{P_{(2,0)}(x)}{P_{(2,0)}(t^\delta)} = \frac{1-qt}{(1+t)(1-qt^2)}(x_1^2+x_2^2) +\frac{(1+q)(1-t)}{(1+t)(1-qt^2)}x_1x_2,
		\end{align*}
		so 
		\begin{align*}
			\frac{P_{(2,0)}(x)}{P_{(2,0)}(t^\delta)}-\frac{P_{(1,1)}(x)}{P_{(1,1)}(t^\delta)} = \frac{(1-qt)}{t(t+1)(1-qt^2)} (tx_1-x_2)(x_1-tx_2).
		\end{align*}
		Assuming $q,t\in(0,1)$ and $x_1,x_2\in(0,\infty)$, then the difference is negative if $t<\frac{x_1}{x_2}<\frac{1}{t}$ and non-negative otherwise.
		
		Numerical experiments suggest that for $d\geq3$, the difference $\frac{P_{(d,0)}(x)}{P_{(d,0)}(t^\delta)}-\frac{P_{(d-1,1)}(x)}{P_{(d-1,1)}(t^\delta)}$ behaves in the same way: assuming $q,t\in(0,1)$ and $x_1,x_2\in(0,\infty)$, then the difference is negative if $t<\frac{x_1}{x_2}<\frac{1}{t}$ and non-negative otherwise.
	\end{example}
	
	Given the above example, we see that the natural choice of normalization, $x=t^\delta$, is not good for positivity. Instead, the previous choice for Jack polynomials, $x=\bm1_n$, will lead to a certain positivity, as demonstrated below.

	\begin{example}\label{example:Mac}
		Continue with $n=2$ and consider the difference $\frac{P_\lambda(x_1,x_2)}{P_\lambda(1,1)} -\frac{P_\mu(x_1,x_2)}{P_\mu(1,1)}$. Now
		\begin{align*}
			P_\lambda(1,1) =	\sum_{i=0}^d \frac{\qpoch{q}{q}{d}}{\qpoch{q}{q}{i}\qpoch{q}{q}{d-i}} \frac{\qpoch{t}{q}{i}\qpoch{t}{q}{d-i}}{\qpoch{t}{q}{d}}
			=	\frac{\qpoch{q}{q}{d}}{\qpoch{t}{q}{d}} \sum_{i=0}^d \frac{\qpoch{t}{q}{i}}{\qpoch{q}{q}{i}} \frac{\qpoch{t}{q}{d-i}}{\qpoch{q}{q}{d-i}}.
		\end{align*}
		Hence we have 
		\begin{align*} 
			\frac{P_\lambda(x_1,x_2)}{P_\lambda(1,1)} = \frac{\displaystyle\sum_{i=0}^d \frac{\qpoch{t}{q}{i}}{\qpoch{q}{q}{i}} \frac{\qpoch{t}{q}{d-i}}{\qpoch{q}{q}{d-i}} x_1^i x_2^{d-i}}{\displaystyle\sum_{i=0}^d \frac{\qpoch{t}{q}{i}}{\qpoch{q}{q}{i}} \frac{\qpoch{t}{q}{d-i}}{\qpoch{q}{q}{d-i}} }
			\text{\quad and \quad}
			\frac{P_\mu(x_1,x_2)}{P_\mu(1,1)} = \frac{\displaystyle\sum_{i=1}^{d-1} \frac{\qpoch{t}{q}{i-1}}{\qpoch{q}{q}{i-1}} \frac{\qpoch{t}{q}{d-1-i}}{\qpoch{q}{q}{d-1-i}} x_1^{i} x_2^{d-i}}{\displaystyle\sum_{i=1}^{d-1} \frac{\qpoch{t}{q}{i-1}}{\qpoch{q}{q}{i-1}} \frac{\qpoch{t}{q}{d-1-i}}{\qpoch{q}{q}{d-1-i}} }.
		\end{align*}
		Their difference is 
		\begin{align*}
			&\=	\frac{P_\lambda(x_1,x_2)}{P_\lambda(1,1)}-\frac{P_\mu(x_1,x_2)}{P_\mu(1,1)}
			\\&=	\frac{\displaystyle\sum_{i=0}^d \frac{\qpoch{t}{q}{i}}{\qpoch{q}{q}{i}} \frac{\qpoch{t}{q}{d-i}}{\qpoch{q}{q}{d-i}} x_1^i x_2^{d-i}}{\displaystyle\sum_{i=0}^d \frac{\qpoch{t}{q}{i}}{\qpoch{q}{q}{i}} \frac{\qpoch{t}{q}{d-i}}{\qpoch{q}{q}{d-i}} } - \frac{\displaystyle\sum_{i=1}^{d-1} \frac{\qpoch{t}{q}{i-1}}{\qpoch{q}{q}{i-1}} \frac{\qpoch{t}{q}{d-1-i}}{\qpoch{q}{q}{d-1-i}} x_1^{i} x_2^{d-i}}{\displaystyle\sum_{i=1}^{d-1} \frac{\qpoch{t}{q}{i-1}}{\qpoch{q}{q}{i-1}} \frac{\qpoch{t}{q}{d-1-i}}{\qpoch{q}{q}{d-1-i}} }
			\\&\propto	\sum_{j=1}^{d-1}\frac{\qpoch{t}{q}{j-1}}{\qpoch{q}{q}{j-1}} \frac{\qpoch{t}{q}{d-1-j}}{\qpoch{q}{q}{d-1-j}}	\cdot	\sum_{i=0}^d \frac{\qpoch{t}{q}{i}}{\qpoch{q}{q}{i}} \frac{\qpoch{t}{q}{d-i}}{\qpoch{q}{q}{d-i}} x_1^i x_2^{d-i}
			\\&\=	-\sum_{j=0}^d \frac{\qpoch{t}{q}{j}}{\qpoch{q}{q}{j}} \frac{\qpoch{t}{q}{d-j}}{\qpoch{q}{q}{d-j}}	\cdot	\sum_{i=1}^{d-1} \frac{\qpoch{t}{q}{i-1}}{\qpoch{q}{q}{i-1}} \frac{\qpoch{t}{q}{d-1-i}}{\qpoch{q}{q}{d-1-i}} x_1^{i} x_2^{d-i}
			\\&=	\sum_{j=1}^{d-1}\frac{\qpoch{t}{q}{j-1}}{\qpoch{q}{q}{j-1}}\frac{\qpoch{t}{q}{d-1-j}}{\qpoch{q}{q}{d-1-j}} \cdot \frac{\qpoch{t}{q}{d}}{\qpoch{q}{q}{d}} (x_1^d+x_2^d) 
			+\sum_{i=1}^{d-1} \left(\frac{\qpoch{t}{q}{i}}{\qpoch{q}{q}{i}} \frac{\qpoch{t}{q}{d-i}}{\qpoch{q}{q}{d-i}} \cdot \sum_{j=1}^{d-1}  \frac{\qpoch{t}{q}{j-1}}{\qpoch{q}{q}{j-1}} \frac{\qpoch{t}{q}{d-1-j}}{\qpoch{q}{q}{d-1-j}} \right.
			\\&\=	 \left. -\frac{\qpoch{t}{q}{i-1}}{\qpoch{q}{q}{i-1}} \frac{\qpoch{t}{q}{d-1-i}}{\qpoch{q}{q}{d-1-i}} \cdot \sum_{j=0}^d  \frac{\qpoch{t}{q}{j}}{\qpoch{q}{q}{j}} \frac{\qpoch{t}{q}{d-j}}{\qpoch{q}{q}{d-j}}\right) x_1^i x_2^{d-i}.
		\end{align*}
		
		Define $a_i$ to be the coefficient such that (via Abel's lemma)
		\begin{align*}
			\frac{P_\lambda(x)}{P_\lambda(\bm1)} - \frac{P_\mu(x)}{P_\mu(\bm1)} 
			=	\sum_{i=0}^{\lfloor\frac{d}{2}\rfloor} a_i M_{(d-i,i)}(x)
			=	\sum_{l=0}^{\lfloor\frac{d}{2}\rfloor-1} \left(\sum_{i=0}^l a_i\right) \left( M_{(d-l,l)}(x)-M_{(d-l-1,l+1)}(x)\right).
		\end{align*}
		Note that the sum $a_0+\cdots+a_{\lfloor\frac{d}{2}\rfloor}=0$, by evaluating at $x=(1,1)$.
		
		For $k\geq1$, let 
		\begin{align*}
			b_k\coloneqq \dfrac{1-tq^{k-1}}{1-q^k},\quad 
			c_k\coloneqq \dfrac{\qpoch{t}{q}{k}}{\qpoch{q}{q}{k}}=\prod_{i=1}^k b_i.
		\end{align*}
		Then we have 
		\begin{align*}
			a_0	&\coloneqq	2c_d \sum_{j=1}^{d-1}c_{j-1}c_{d-j-1},	\\
			a_i	&\coloneqq	2c_ic_{d-i} \sum_{j=1}^{d-1} c_{j-1}c_{d-j-1} -2c_{i-1}c_{d-i-1} \sum_{j=0}^d c_{j}c_{d-j},\quad	1\leq i\leq \lfloor\frac{d-1}2\rfloor,	\\
			a_{d/2}	&\coloneqq	\left(c_{d/2}\right)^2 \sum_{j=1}^{d-1} c_{j-1}c_{d-j-1} 
			-\left(c_{d/2-1}\right)^2 \sum_{j=0}^{d}  c_j c_{d-j},\quad\text{if $d$ is even}.
		\end{align*}
		
		In order to ``upgrade'' \cref{conj:Muirhead} from Jack to Macdonald polynomials, it suffices to show that $\sum_{i=0}^l a_i\geq0$, for $1\leq l\leq \lfloor d/2\rfloor-1$.
		
		\begin{lemma}
			Let $q,t\in(0,1)$.
			\begin{enumerate}
				\item The sequence $(b_i)_{i\geq1}$ is positive and monotone.
				\item For $1\leq l\leq \lfloor (d-1)/2\rfloor$, and $l<s_i<d-l$ for $i=1\dots,l$, the following $(2l+1)$-term polynomial is positive.
				\begin{align}
					&\=	\sum_{i=1}^l \left(\prod_{m=1}^{i-1} b_m \cdot (b_{i}b_{d-i}-b_{s_i}b_{d-s_i}) \cdot \prod_{m=d-l}^{d-i-1} b_{m}\right) + \prod_{m=d-l}^d b_m	\notag
					\\&=	\prod_{m=1}^l b_m\cdot b_{d-l} + 
					\sum_{i=1}^l \left(\prod_{m=1}^{i-1} b_m \cdot (b_{d-i}b_{d-i+1}-b_{s_i}b_{d-s_i}) \cdot \prod_{m=d-l}^{d-i-1} b_{m}\right). \label{eqn:2l+1}
				\end{align}
			\end{enumerate}
		\end{lemma}
		
		\begin{proof}
			(1) It is clear that $b_i>0$. Note that
			\begin{align*}
				b_{i+1}-b_i
				=	\dfrac{1-tq^{i}}{1-q^{i+1}} -\dfrac{1-tq^{i-1}}{1-q^i}
				=	\frac{q^{i-1}(t-q)(1-q)}{(1-q^i)(1-q^{i+1})}.
			\end{align*}
			Hence $(b_i)$ is increasing if $q<t$ and decreasing if $q>t$.
			
			(2)
			We first show that the two expressions are equal. Separate the positive term for $i=l$, and collect the negative term for $i$ and the positive term for $i-1$; now the last term above becomes the positive term for $i=1$.
			If $q=t$, the $(2l+1)$-term polynomial becomes $1$.
			If $q<t$, then by part (1), $b_{d-i}b_{d-i+1}-b_{s_i}b_{d-s_i}>0$.
			If $q>t$, then 
			\begin{align*}
				b_ib_{d-i}-b_{s_i}b_{d-s_i} 
				&=	\frac{1-tq^{i-1}}{1-q^{i}}\frac{1-tq^{d-i-1}}{1-q^{d-i}}-\frac{1-tq^{s_i-1}}{1-q^{s_i}}\frac{1-tq^{d-s_i-1}}{1-q^{d-s_i}}
				\\&=	\frac{q^{i-1}(q-t)(1-tq^{d-1})(1-q^{s_i-i})(1-q^{d-i-s_i})}{(1-q^{i})(1-q^{d-i-1})(1-q^{s_i})(1-q^{d-s_i-1})}>0. \qedhere
			\end{align*}
		\end{proof}
		
		Now, we prove $\sum_{i=0}^l a_i\geq0$, for $1\leq l\leq \lfloor d/2\rfloor-1$.
		
		$\bullet$ Let $d=2k$, $k\geq1$. We have 
		\begin{align*}	
			a_0	&\coloneqq	4c_d\sum_{j=1}^{k-1}c_{j-1}c_{d-j-1} + 2c_{k-1}^2c_d,	\\
			a_i	&\coloneqq	4c_ic_{d-i}\sum_{j=1}^{k-1} c_{j-1}c_{d-j-1} -4c_{i-1}c_{d-i-1} \sum_{j=1}^{k-1} c_{j}c_{d-j}-4c_{i-1}c_{d-i-1}c_d
			\\&\=	+ 2c_ic_{k-1}^2c_{d-i}-2c_{i-1}c_{k}^2c_{d-i-1},\quad	1\leq i\leq k-1,	\\
			a_{k}	&\coloneqq	 2c_{k}^2\sum_{j=1}^{k-1} c_{j-1}c_{d-j-1}
			-2c_{k-1}^2 \sum_{j=1}^{k-1}c_j c_{d-j}-2c_{k-1}^2c_d.
		\end{align*}
		For $1\leq l\leq k-1$, we have
		\begin{align*}
			&\=	a_0+\cdots+a_l 
			\\&=	4c_d\sum_{j=1}^{k-1}c_{j-1}c_{d-j-1} + 2c_{k-1}^2c_d
			\\&\=	+\sum_{i=1}^l 4c_ic_{d-i}\sum_{j=1}^{k-1} c_{j-1}c_{d-j-1} -\sum_{i=1}^l4c_{i-1}c_{d-i-1} \sum_{j=1}^{k-1} c_{j}c_{d-j}-\sum_{i=1}^l4c_{i-1}c_{d-i-1}c_d
			\\&\=	+ \sum_{i=1}^l2c_ic_{k-1}^2c_{d-i} -\sum_{i=1}^l2c_{i-1}c_{k}^2c_{d-i-1}
			\\&=	4c_d\sum_{j=l+1}^{k-1}c_{j-1}c_{d-j-1} 
			+4\sum_{i=1}^l c_ic_{d-i} \sum_{j=l+1}^{k-1} c_{j-1}c_{d-j-1} -4\sum_{i=1}^l c_{i-1}c_{d-i-1} \sum_{j=l+1}^{k-1} c_{j}c_{d-j}
			\\&\=	+ 2c_{k-1}^2c_d + 2\sum_{i=1}^lc_ic_{k-1}^2c_{d-i} -2\sum_{i=1}^lc_{i-1}c_{k}^2c_{d-i-1},\quad	1\leq i\leq k-1,	
			\\&=	4\sum_{j=l+1}^{k-1} \left(c_{j-1}c_{d-j-1}c_d+\sum_{i=1}^l \left(c_ic_{d-i}c_{j-1}c_{d-j-1} - c_{i-1}c_{d-i-1}c_{j}c_{d-j}\right)\right)
			\\&\=	+2\left(c_{k-1}^2c_d+\sum_{i=1}^l \left(c_ic_{k-1}^2c_{d-i} - c_{i-1}c_{k}^2c_{d-i-1}\right)\right)
			\\&=	4 \sum_{j=l+1}^{k-1} c_{j-1}c_{d-j-1}c_{d-l-1} \left(\prod_{m=d-l}^{d}b_m+\sum_{i=1}^l \prod_{m=1}^{i-1}b_m\cdot  \left(b_ib_{d-i} -b_{j}b_{d-j}\right)\cdot \prod_{m=d-l}^{d-i-1}b_m\right)
			\\&\=	+2c_{k-1}^2 c_{d-l-1}\left( \prod_{m=d-l}^{d}b_m + \sum_{i=1}^l \prod_{m=1}^{i-1}b_m\cdot \left(b_ib_{d-i} - b_{k}^2\right)\cdot \prod_{m=d-l}^{d-i-1}b_m \right).
		\end{align*}
		Each expression enclosed in parentheses is of the form in \cref{eqn:2l+1}, hence is positive.
		
		$\bullet$ Let $d=2k+1$, $k\geq1$. We have
		\begin{align*}
			a_0	&\coloneqq	4c_d \sum_{j=1}^{k}c_{j-1}c_{d-j-1},	\\
			a_i	&\coloneqq	4c_ic_{d-i} \sum_{j=1}^{k} c_{j-1}c_{d-j-1} -4c_{i-1}c_{d-i-1} \sum_{j=1}^{k} c_{j}c_{d-j}-4c_{i-1}c_{d-i-1}c_d,\quad	1\leq i\leq k.
		\end{align*}
		For $1\leq l\leq k$, we have
		\begin{align*}
			&\=	\frac14(a_0+\cdots+a_l)
			\\&=	\sum_{j=l+1}^{k} c_{j-1}c_{d-j-1}c_{d} +\sum_{i=1}^l c_ic_{d-i} \sum_{j=l+1}^{k} c_{j-1}c_{d-j-1} -\sum_{i=1}^{l} c_{i-1}c_{d-i-1} \sum_{j=l+1}^{k} c_{j}c_{d-j}
			\\&=	\sum_{j=l+1}^{k} \left(c_{j-1}c_{d-j-1}c_{d} +\sum_{i=1}^l \left(c_ic_{d-i}c_{j-1}c_{d-j-1} -c_{i-1}c_{d-i-1}c_{j}c_{d-j}\right)\right)
			\\&=	\sum_{j=l+1}^{k} c_{j-1}c_{d-j-1}c_{d-l-1} \left(\prod_{m=d-l}^{d}b_m +\sum_{i=1}^l \prod_{m=1}^{i-1}b_m\cdot \left(b_ib_{d-i}-b_{j}b_{d-j}\right)\cdot \prod_{m=d-l}^{d-i-1}b_m\right).
		\end{align*}
		As the expression within each parentheses is of the form in \cref{eqn:2l+1}, it is positive. 
	\end{example}
	
	Thus, we arrive at the following conjecture for normalized Macdonald
	differences in the Muirhead semiring---and we now merge this with the
	Macdonald analogue of the Cuttler--Greene--Skandera conjecture too.
	Let $\kp$ and $\kp^\R$ be the cones for the Macdonald case, as defined in \cref{eqn:fp-Mac}.
	
	\begin{conj}[CGS and Muirhead conjectures for Macdonald polynomials]\label{conj:CGS-Muirhead-Mac}
		The following are equivalent for partitions $\lambda$ and $\mu$ :
		\begin{enumerate}
			\item The Macdonald difference lies in the Muirhead semiring over $\kp$:
			\begin{equation}
				\frac{P_\lambda(x;q,t)}{P_\lambda(\bm1;q,t)} - \frac{P_\mu(x;q,t)}{P_\mu(\bm1;q,t)}\in \mathcal M_S(\kp).
			\end{equation}
			\item We have
			\begin{align}
				\frac{P_\lambda(x;q,t)}{P_\lambda(\bm1;q,t)} -
				\frac{P_\mu(x;q,t)}{P_\mu(\bm1;q,t)}\in\kp^{\mathbb R}, \quad
				\forall x\in[0,\infty)^n.
			\end{align}
			
			\item[$(2')$] For every $q_0,t_0\in(0,1)$, the Macdonald difference lies in
			the Muirhead semiring over $\mathbb{R}_{\geqslant 0}$:
			\begin{equation}
				\frac{P_\lambda(x;q,t)}{P_\lambda(\bm1;q,t)} - \frac{P_\mu(x;q,t)}{P_\mu(\bm1;q,t)}\in \mathcal M_S(\mathbb{R}_{\geqslant 0}).
			\end{equation}
			
			\item For every $q_0, t_0 \in (0,1)$, we have
			\begin{align}
				\frac{P_\lambda(x;q_0,t_0)}{P_\lambda(\bm1;q_0,t_0)} -
				\frac{P_\mu(x;q_0,t_0)}{P_\mu(\bm1;q_0,t_0)}\geqslant 0, \quad
				\forall x\in[0,\infty)^n.
			\end{align}
			
			\item For some $q_0, t_0 \in (0,1)$, we have
			\begin{align}
				\frac{P_\lambda(x;q_0,t_0)}{P_\lambda(\bm1;q_0,t_0)} -
				\frac{P_\mu(x;q_0,t_0)}{P_\mu(\bm1;q_0,t_0)}\geqslant 0, \quad
				\forall x\in(0,1)^n \cup (1,\infty)^n.
			\end{align}
			
			\item $\lambda$ majorizes $\mu$.
		\end{enumerate}
	\end{conj}
	
	\begin{remark}\label{remark:duality}
		By the well-known duality $P_\lambda(x;q,t) = P_\lambda(x;1/q,1/t)$, one can define instead $\kp\coloneqq\Set{f\in\Q(q,t)}{f(q,t)\geq0,\text{\ if\ } q,t\in(1,\infty)}$ and get an equivalent statement.
	\end{remark}
	
	As promised, we now prove \cref{conj:CGS-Muirhead-Mac} in the cases where
	we have shown it for Jack polynomials above.
	
	\begin{theorem}\label{thm:CGS-Muirhead-Mac}
		In \cref{conj:CGS-Muirhead-Mac}, it is clear that $(1)\Rightarrow(2)\&(2')$, $(2)\Rightarrow(3)$, $(2')\Rightarrow(3)$, and $(3)\Rightarrow(4)\Rightarrow(5)$.
		Hence the conjecture reduces to the implication $(5)\Rightarrow(1)$.
		Moreover, this implication holds for $\mu=(1^{|\lambda|})$; and it also holds for arbitrary $\lambda$ and $\mu$ in the case of two variables.
	\end{theorem}
	
	\begin{proof}
		Note that 
		$(1) \Rightarrow (2)$ and $(2')\Rightarrow(3)$ follows from Muirhead's inequality; 
		$(1)\Rightarrow(2')$ and $(2)\Rightarrow(3)$ follow by specializing $q,t$;
		$(3) \Rightarrow (4)$ is obvious;
		$(4) \Rightarrow (5)$ is similar to \cref{thm:CGSimpliesKT} and is omitted.
		
		Next, we assume $\mu = (1^{|\lambda|})$, with $n, \lambda$ arbitrary and show $(5) \Rightarrow (1)$. Haglund--Haiman--Loehr have shown
		\cite{HHL05} that the integral Macdonald polynomials are
		monomial-positive over $\kp$ for $q,t \in (0,1)$, whence so are the $P_\lambda$. 
		Now the calculation in \cref{eqn:Muirhead-cone} implies~(1).
		
		Finally, let $n=2$ and $\lambda, \mu$ have length at most 2. To show $(5) \Rightarrow (1)$, it suffices via telescoping to do so for majorization-adjacent
		partitions $\lambda=(d-i,i)$ and $\mu=(d-i-1,i+1)$. We showed this in
		\cref{example:Mac} for $i=0$; and we use this to compute for $i>0$:
		\begin{align*}
			\frac{P_\lambda(x;q,t)}{P_\lambda(\bm1;q,t)} - \frac{P_\mu(x;q,t)}{P_\mu(\bm1;q,t)} 
			&=	(x_1x_2)^i \left(\frac{P_{(d-2i,0)}(x;q,t)}{P_{(d-2i,0)}(\bm1;q,t)} - \frac{P_{(d-2i-1,1)}(x;q,t)}{P_{(d-2i-1,1)}(\bm1;q,t)} \right)
			\\&\in	(x_1x_2)^i\cdot \sum_{\nu\m\xi} \kp\cdot(M_\nu(x)-M_\xi(x))
			\\&=	\sum_{\nu+(i,i)\m\xi+(i,i)} \kp\cdot(M_{\nu+(i,i)}(x)-M_{\xi+(i,i)}(x))
			\subset \mathcal M_C(\kp).
			\qedhere
		\end{align*} 
	\end{proof}

	\begin{remark}
		For two variables, the above proof reveals that
		\cref{conj:CGS-Muirhead-Mac}(5) implies the stronger (than~(1) \textit{a
			priori}, and hence equivalent) statement that the normalized Macdonald
		difference for $\lambda, \mu$ lies in the Muirhead \textit{cone}, just
		like for Jack polynomials. As \cref{ex:Muirhead} shows, this is false
		in general for three variables.
	\end{remark}
	
	\subsection*{Acknowledgments}
	S.S.\ and A.K.\ would like to thank IISc Bangalore and IAS Princeton, respectively, for their hospitality during summer 2022 and summer 2023, when this work was initiated and then continued. 
	H.C.\ was partially supported by the Lebowitz Summer Research Fellowship and the SAS Fellowship at Rutgers University. 
	A.K.\ was partially supported by SwarnaJayanti Fellowship grants SB/SJF/2019-20/14 and DST/SJF/MS/2019/3 from SERB and DST (Govt. of India) and by a Shanti Swarup Bhatnagar Award from CSIR (Govt. of India). 
	The research of S.S.\ during this project has been supported by NSF grant DMS-2001537 and Simons Foundation grants 509766 and 00006698.

	\bibliographystyle{alpha}

\end{document}